\documentclass[10pt]{article}
\usepackage{amsmath}
\usepackage[titletoc,title]{appendix}
\usepackage{amsbsy}
\usepackage{amssymb}
\usepackage{amsthm}
\usepackage{anyfontsize}
\usepackage{wrapfig}
\usepackage{color,cancel}
\usepackage{fullpage}
\usepackage{float}
\usepackage{ulem}
\usepackage{breqn}
\usepackage{graphicx}
\usepackage{caption}
\usepackage{subcaption}
\usepackage{bbm}
\usepackage{calrsfs}
\usepackage{verbatim}
\usepackage{url} 
\pdfoutput=1

\usepackage{amsmath,amsfonts,amssymb,color,graphicx}

\newtheorem{theorem}{Theorem}[section]
\newtheorem{Algorithm}{Algorithm}[section]
\newtheorem{corollary}{Corollary}[section]
\newtheorem{lemma}{Lemma}[section]

\newtheorem{definition}{Definition}[section]
\newtheorem{remark}{Remark}[section]

\numberwithin{equation}{section}
\newcommand{\norm}[1]{\left\Vert#1\right\Vert}

\newcommand{\bu}{{\bf u}}
\newcommand{\bv}{{\bf v}}
\newcommand{\bw}{{\bf w}}
\newcommand{\bx}{{\bf x}}
\newcommand{\be}{{\bf e}}
\newcommand{\bff}{{\bf f}}

\newcommand{\bphi}{{\boldsymbol \phi}}

\newcommand{\bfX}{{\bf X}}
\newcommand{\bfV}{{\bf V}}

\newcommand{\bfeta}{{\boldsymbol \eta}}

\newcommand{\bchi}{{\boldsymbol \chi}}

\usepackage{color}
\usepackage[section]{placeins}
\makeatletter
\AtBeginDocument{%
	\expandafter\renewcommand\expandafter\subsection\expandafter{%
		\expandafter\@fb@secFB\subsection
	}%
}

\date{}

\title{An analysis of a linearly extrapolated BDF2 subgrid artificial viscosity method for incompressible flows  }

\author{
Medine Demir
\thanks{Department of Mathematics, Middle East Technical University, 06800 Ankara, Turkey; dmedine@metu.edu.tr}
\and
Song\"{u}l Kaya
\thanks{
Department of Mathematics, Middle East Technical University, 06800 Ankara, Turkey; smerdan@metu.edu.tr.}
}

\begin{document}

\maketitle

\textbf{Abstract.} This report extends the mathematical support of a subgrid artificial viscosity (SAV) method to simulate the incompressible Navier-Stokes equations to better performing a linearly extrapolated BDF2 (BDF2LE) time discretization. The method considers the viscous term as a combination of the vorticity and the grad-div stabilization term. SAV method introduces global stabilization by adding a term, then anti-diffuses through the extra mixed variables. We present a detailed analysis of conservation laws, including both energy and helicity balance of the method. We also show that the approximate solutions of the method are unconditionally stable and optimally convergent. Several numerical tests are presented for validating the support of the derived theoretical results.





\section{Introduction}


Incompressible viscous fluid flows are expressed by the Navier-Stokes equations (NSE), which are given as follows:
\begin{equation}\label{nse}
\begin {array}{rcll}
\bu_t -\nu\triangle \bu+\bu\cdot \nabla \bu+ \nabla p &=& \bf f &
\mathrm{in }\ \Omega\times (0,T), \\
\nabla \cdot \bu&=& 0& \mathrm{in }\ \Omega\times  (0,T],\\
\bu&=& \mathbf{0}& \mathrm{in }\ \partial\Omega\times[0,T],\\
\bu(\bx,0)&=& \bu_{0}(\bx) & \mathrm{for}\  \bx \in \Omega.
\end{array}
\end{equation}
Here $\bu$ represents the velocity, $p$ the zero-mean pressure, $\bff$ an external force and $\nu$ the kinematic viscosity. It is well known that when Reynolds number gets higher, the range of scales expands, as a result computational cost also increases too much. Thus, a direct discretization of equation (\ref{nse}) such as by the Galerkin finite element method, remains incapable to simulate turbulent flows. One successful approach to model NSE is variational multiscale (VMS) methods, whose aim is to design stabilized finite element methods. The main idea of VMS includes defining large scales with projection into appropriate function spaces. VMS methods have been developed by \cite{HMJ,JLG} with the motivation of the inherent multiscale structure of the solution.  There are various classes and realizations of VMS for different types of fluid problems, see \cite{V17} for an overview. 


Due to the proven good theoretical and practical properties of VMS methods, it is natural to broaden its understanding by developing efficient, accurate and stable numerical algorithms. The method we consider in this paper is first proposed in \cite{L} which is in fact VMS method, for finding solutions to the convection-dominated convection diffusion equation. The VMS method of \cite{L} introduces global stabilization by adding a term, then anti-diffuses through the extra mixed variables which are chosen as the large scales of solution. Since the effective artificial viscosity type stabilization influences only the small scales, it can be thought as a subgrid artificial viscosity (SAV) method. Based on these ideas, in page 156 of \cite{L}, the new formulation of SAV has been proposed without any numerical analysis.  In this formulation, the stability process is applied to the viscous term by using the vector identity $\Delta \bu=-\nabla\times(\nabla\times \bu)+\nabla(\nabla\cdot{\bu})$ to reduce the extra storage in $3d$. As a result, a two-level method is obtained that combines both vorticity and the grad-div stabilization in the viscous term. This SAV method greatly reduces extra storage compared to velocity and its gradient. SAV method was first analyzed in the study of \cite{G} by using Crank-Nicholson (CN) scheme. As noted in \cite{G}, the system is improved in a better way without choosing computationally inefficient time step, that is, the system of method includes just three variables with the use of coarse grid of vorticity instead of nine variables for the full velocity gradient. Thus, the method significantly improves the solution of the system in case of a small viscosity without choosing computationally inefficient time step. The current paper extends the mathematical idea of mixed, conforming SAV finite element method to the  multistep BDF2 times discretization. Since it exhibits strong stability and damping properties that are better than those of CN for the simulation of underresolved regimes, BDF2 is one of the most popular choice of time stepping scheme \cite{28}. We note that the backward differentiation formula is a class of time-stepping scheme which has been commonly used and studied for the time dependent ordinary and partial differential equations \cite{ 20, 21, 19, 22, 24, 26, 25, 23, 27}. In light of the previous literature, herein we consider a successful SAV stabilization scheme to be used with linearly extrapolated BDF2 (BDF2LE) formula in time without hurting accuracy. 

This paper carefully considers several physical and mathematical questions concerning SAV solutions, and it is arranged as follows. In Section 4, we first investigate the conservation of the fundamental integral variants of fluid flow energy and helicity for SAV solutions. It is well known these quantities are important for the  physical fidelity of the model, but most models do not conserve them. In most common Galerkin finite element discretization for incompressible flow problems energy conservation is lost. To preserve the conservation of energy, the skew-symmetric or rotational formulations of the nonlinear term is used, \cite{TAM95}. A finite element formulation of \cite{OR} and \cite{L07} preserves energy and helicity with a slightly altered definition of discrete helicity together with the skew-symmetric formulation of nonlinear term. In this report, by using underlying ideas of  \cite{OR, L07}, we show that without viscous or external forces, the energy and helicity will remain constant in time and they will be correctly balanced  when these forces are present.

Section 5 gives a complete numerical analysis of SAV method along with the proofs of unconditional stability and convergence. In our scheme, overall accuracy and mass conservation in the discrete solution depend on the carefully chosen stabilization parameters, namely the artificial viscosity and the grad-div stabilization. Standard error analysis for SAV method predicts that the optimal choice for the artificial viscosity parameter should be $O(h)$ and the grad-div stabilization parameter should be $O(1)$.


Section 6 presents several numerical examples in order to present evidence of optimal accuracy for an analytical test problem, and also demonstrate the ability of SAV method to give good results on flow around a cylinder and flow between two offset circles.

\section{BDF2LE Based SAV Method}

In this study, we consider  standard notations for the function spaces, e.g., see Adams \cite{A75}. We assume $\Omega \subset \mathbb{R}^{d} $, $d$ $\in$ $\{2,3\}$, be an open regular domain with boundary  $\partial\Omega$. We denote the $L^2(\Omega)$ norm and inner product by $\norm{\cdot}$ and $(\cdot, \cdot)$, respectively. 

For the weak formulation of (\ref{nse}), we set $$\bfX:=(H_0^1(\Omega))^{d}:=\{\bv\in H^1(\Omega): \bv=0 \ \text{on} \ \partial\Omega \}$$ for the space of vector field, whereas for the pressure we set   
$$Q:=L_0^2(\Omega):=\{q\in L^2(\Omega): \int_{\Omega} q d\bx=0 \}.$$ 
We consider the weak formulation: Find $\bu: (0,T] \longrightarrow \bfX$, $p: (0,T] \longrightarrow Q$ such that
\begin{eqnarray}
(\bu_t,\bv)+\nu(\nabla\bu,\nabla\bv)+b(\bu,\bu,\bv)-(p,\nabla\cdot\bv)&=&(\bff,\bv)\quad \forall \bv \in \bfX, \label{vf}\\
(q,\nabla\cdot\bu)&=&0 \quad \forall q \in Q,\label{vf1} 
\end{eqnarray}
with $\bu (\bf x,0)=\bu_0(\bf x) $ $\in \bfX$. Here, we use the skew-symmetric form of the nonlinear term is
\begin{eqnarray}
b(\bu,\bv,\bw) &=&
\frac{1}{2}\big((\bu\cdot\nabla\bv,\bw)-(\bu\cdot\nabla\bw,\bv)\big) \label{ss1}
\end{eqnarray}
and recall from \cite{Layton} that
\begin{eqnarray}
b(\bu,\bv,\bv)&=&0,\label{d1} \\
 b(\bu,\bv,\bw)&=&-b(\bu,\bw,\bv).\label{d2}
 \end{eqnarray}

For finite element approximation of \eqref{vf}-\eqref{vf1}, let $\bfX_h \subset \bfX$ and $Q_h\subset Q$ be two conforming finite element spaces defined on a fine mesh $\pi_h$ for which the usual discrete inf-sup condition is satisfied, e.g., there is a constant $\beta$ independent of the mesh size $h$ such that
\begin{eqnarray}
\inf_{q_h\in{Q}_h}\sup_{\bv_h\in {\bfX}_h}\frac{(q_h,\,\nabla\cdot
	\bv_h)}{||\,\nabla \bv_h\,||\,||\,q_h\,||}\geq \beta > 0.
\label{infsup}
\end{eqnarray}

To formulate SAV discretization method, we need some further notations. Let  $\pi^H$ be a family of  triangulations of $\Omega$ and let $\pi^h$ be a refinement of  $\pi^H$.  We also introduce a coarse or large scale space $ L_{H}\subset L^{2}(\Omega)^{d}$ on a regular mesh $\pi^H$. Note that the choice of the large scale space is critical for SAV formulation \cite{VS, VKL}. This choice can be done in two different ways. The first choice is to introduce $L_H$ on a coarser grid than $(X_h,Q_h)$. The second choice, which we will use, is to define $ L_H$ on the same grid as $(X_h,Q_h)$ using lower degree polynomials. In our case, we consider piecewise polynomials with degree $k$ for velocity space thus for $L_H$ on the coarse mesh, we use piecewise polynomials with degree $k-1$. For the proposed method, the size of coarse mesh $H$ is also selected such that it does not spoil the optimal asymptotic convergence rate. For SAV based BDF2LE, for $k=2$, the coarse mesh size is chosen as $H={O}({h}^{1/2})$.

We now approximate the formulation of (\ref{vf})-(\ref{vf1}) of the NSE by a second-order accurate SAV algorithm based on BDF2LE by the following algorithm. In this method, for the linear term, implicit time discretization and for the nonlinear term two step linear extrapolation have been used.  Let a positive number 
$N$  be given and define the step size $\Delta t=T/n$, $n=0,1,,...,N$ with given at time  $t^n=n\Delta t$  as follows.
\begin{Algorithm}\label{Alg1}
Let $S_H$ be the new coarse mesh variable and the initial condition $\bu_{0}$ be given. Define $\bu^{0}_h, \bu_h^{-1}$ as the nodal interpolants of $\bu^0$. Then, given $\bu_h^n$, $\bu_h^{n-1}$, $p_h^n$,$p_h^{n-1}$, find  
	$(\bu_{h}^{n+1},{p}_{h}^{n+1}, S_H^{n+1}) \in (\bfX_{h},Q_{h},  L_H)$ satisfying  $\forall (\bv_{h},q_{h},l_{H}) \in (\bfX_{h},Q_{h}, L_{H})$  
	 \begin{eqnarray}
	\bigg(\dfrac{3\bu_{h}^{n+1}-4\bu_{h}^{n}
		+\bu_{h}^{n-1}}{2\Delta t},\bv_{h}\bigg)+\nu(\nabla \bu_{h}^{n+1},\nabla \bv_{h})+b(2\bu_{h}^{n}-\bu_h^{n-1},\bu_{h}^{n+1},\bv_{h})-(p_{h}^{n+1},\nabla \cdot \bv_{h})\nonumber\\
	+\alpha_{1}(\nabla \times \bu_{h}^{n+1},\nabla \times \bv_{h})
	-\alpha_{1}(S_{H}^{n+1},\nabla \times \bv_{h})+
	\alpha_{2}(\nabla \cdot \bu_{h}^{n+1},\nabla \cdot \bv_{h})=(\bff(t^{n+1}),\bv_{h}), \label{1}
	\\
	(\nabla \cdot \bu_{h}^{n+1},q_{h})=0,\label{2}
	\\ 
	(S_{H}^{n+1}-\nabla \times \bu_{h}^{n},l_{H})=0. \label{3}
	\end{eqnarray}
\end{Algorithm}
Herein, $\alpha_{1}=\alpha_{1}(\bx,h)$  is the artificial (subgrid) viscosity parameter and $\alpha_2$ is the grad-div stabilization parameter.
\begin{remark}
The artificial viscosity parameter $\alpha_{1}$ is chosen to be constant in element-wise. 
\end{remark}
\begin{remark}
As it is mentioned in \cite{G}, Algorithm \ref{Alg1} requires a coarse grid storage of vorticity with three variables, instead of the full velocity gradient with nine variables of projection-based VMS, see \cite{L}. In addition, the method adds and subtracts the stabilization for consistency but the subtracted term is treated as extra variable in a mixed method. BDF2LE based SAV method is augmented with the grad-div stabilization term adding such term improves conservation of mass in finite element approximation, \cite{FH88}.
\end{remark}
\section{Notations and Mathematical Preliminaries}

We present some notations and  mathematical preliminaries used throughout the paper.  Define the divergence free subspace of $\bfX$ by $\bfV$:
$$\bfV := \{\bv \in \bfX: (\nabla\cdot \bv, q)=0, \quad \forall  q \in  Q\}.$$
The dual space of $\bfX$ is denoted by $H^{-1}$ with norm 
\[\norm{ \bff }_{-1}  = \sup_{0\neq\bv\in{\bfX}}
\frac{|(\bff,\bv)|}{\norm{\nabla \bv }}.
\]
Throughout the paper, we will frequently use Poincar\'e-Friedrichs inequality: there exists a constant $C_{PF}=C_{PF}(\Omega)$ which depends the size of the domain only, such that 
\begin{eqnarray}
\|\bu\| \leq C_{PF}\|\nabla \bu\|,  \quad \forall \bu \in \bfX \nonumber
\end{eqnarray}
We also assume that an inverse inequality is satisfied for a constant that doesn't depend on $h$ such that
\begin{eqnarray}
\norm{\nabla\bv}\leq Ch^{-1}\norm{\bv}, \quad \forall \bv \in \bfX.
\end{eqnarray}
We state the following lemma to bound the trilinear form \eqref{ss1} arising in the analysis.
\begin{lemma}	For $\bu,\bv,\bw$ $\in \bfX$ 
	\begin{eqnarray}
		b(\bu,\bv,\bw) &\leq& C(\Omega)\|\nabla \bu\|\|\nabla \bv\|\|\nabla \bw\|. \label{n3}
		\end{eqnarray}
	In addition, if $\bv, \nabla \bv \in L^{\infty }(\Omega)$,
	\begin{eqnarray}
	b(\bu,\bv,\bw) &\leq&  \dfrac{1}{2}\bigg(\norm{\bu}\|\nabla \bv\|_{\infty}\|\nabla \bw\|+\norm{\bu}\norm{\bv}_{\infty}\|\|\nabla \bw\|\bigg).\label{n2}
	\end{eqnarray}
\end{lemma}
\begin{proof}
For a proof of (\ref{n3}) and (\ref{n2}), use H\"{o}lder's and Ladyzhenskaya inequalities and the Sobolev embedding theorem, see \cite{Layton}. The second inequality follows from H\"{o}lder's and Poincar\'e-Friedrichs' inequalities. 
\end{proof}

In addition, we define the discretely divergence free function space by
\begin{eqnarray}
\bfV_{h}&:=&\{ \bv_{h}\in \bfX_h \ | (\nabla\cdot \bv_{h},q_{h})=0, \ \forall q_{h} \  \in Q_{h}\}.\nonumber
\end{eqnarray}
Under the inf-sup condition (\ref{infsup}), the variational formulation of NSE (\ref{nse}) in $(\bfX_h,Q_h)$ is equivalent to in $(\bfV_h,Q_h)$, see, e.g., \cite{GR79}. We also assume that $(\bfX_h,Q_h)$ satisfy the well-known approximation properties for the choice of typical piecewise polynomials of degree $(k,k-1)$, (see, e.g., \cite{G})
\begin{eqnarray}
\inf_{\bv_h\in {\bfX}_h} \left( \|(\bu-\bv_h)\|+h\|\nabla(\bu-\bv_h)\| \right)&\leq&
C h^{k+1} | \bu |_{k+1}\quad{\forall\bu} \in H^{k+1}(\Omega),\label{ap1}\\
\inf_{q_h\in {Q}_h} \|p-q_h\|&\leq & Ch^k |p|_{k}\quad {\forall p} \in H^{k}(\Omega).\label{ap10} 
\end{eqnarray}
\begin{definition} The $L^2$ projection  $P_{ L_H} : (L^2(\Omega))^{d \times d}\longrightarrow  L_H$ is also defined by 
\begin{eqnarray}
(P_{ L_H}\phi-\phi,l_H)=0 \quad \quad \forall l_H \in  L_H.\label{cm}
\end{eqnarray}
\end{definition}
Then, approximation on coarse mesh space $ L_H$ is given by
\begin{eqnarray}
\norm{\phi-P_{ L_H}\phi}\leq CH^k|\phi|_{k+1},\quad \phi \in L^2(\Omega)\cap (H^{k+1}(\Omega))^d.\label{pa}
\end{eqnarray}
For convergence analysis, we use also the discrete Gronwall's lemma given in \cite{GR79}:
\begin{lemma}(Discrete Gronwall's Lemma)\label{gl}
	Assume that $\Delta t$, H and $a_{n},b_{n},c_{n},d_{n}$ (for integers n$\geq$ 0) be non-negative numbers such that if
	\begin{equation}
	a_{N}+\Delta t\sum_{n=0}^{N} b_{n} \leq \Delta t\sum_{n=0}^{N}d_{n}a_{n}+\Delta t\sum_{n=0}^{N}c_{n}
+H \quad \forall N\geq 1
	\end{equation}
then for $\Delta t>0$,
	\begin{equation}
	a_{N}+\Delta t\sum_{n=0}^{N} b_{n} \leq \exp \Big(\Delta t\sum_{n=0}^{N}\dfrac{d_{n}}{1-\Delta t d_{n}}\Big)\Big(\Delta t\sum_{n=0}^{N}c_{n}
	+H\Big) \quad \text{for} \quad N\geq 0.
	\end{equation}
	
\end{lemma}

\section{Conservation Laws for SAV Solution}

This section studies the discrete conservation laws of Algorithm \ref{Alg1}. We present both energy and helicity balance of the algorithm. In general, helicity is not generally preserved for usual boundary conditions, see \cite{2}. An alternative discrete helicity definition, proposed in \cite{OR}, uses the solution of the discrete vorticity equation instead of being the curl of the velocity. In this way, helicity balance is recovered up to the boundary effect. Following underlying ideas of \cite{OR}, we choose $\widetilde\bfX_h$ to be the vorticity space which is the same as velocity discrete space but without imposing homogeneous Dirichlet boundary condition. Herein, discrete helicity definition for Galerkin discretization of the NSE which is denoted by $H_h(t)$ and computed as
$$H_h(t)=\int_{\Omega}\bu_h(t)\cdot\bw_h(t).$$
Herein $\bw_h$ denotes the solution of a discrete vorticity equation. The discrete vorticity equation is as follows: For given $\bu_h(t)$, for all $t>0$, find $(\bw_h(t),\lambda_h,D_H) \in (\widetilde\bfX_h,Q_h,L_H)$ satisfying $\forall (\chi_h,\tau_h,\rho_H) \in (\bfX_h, Q_h,  L_H )$ 
\begin{eqnarray}
\bigg(\dfrac{3\bw_{h}^{n+1}-4\bw_{h}^{n}+\bw_{h}^{n-1}}{2\Delta t},\bchi_{h}\bigg)+
\nu(\nabla \bw_{h}^{n+1},\nabla \bchi_{h})-b(2\bw_{h}^{n}-\bw_h^{n-1},\bu_{h}^{n+1},\bchi_{h})\nonumber\\+b(2\bu_{h}^{n}-\bu_h^{n-1},\bw_{h}^{n+1},\bchi_{h})-(\lambda_{h}^{n+1},\nabla \cdot \bchi_{h})
+\alpha_{1}(\nabla \times \bw_{h}^{n+1},\nabla \times \bchi_{h})\nonumber\\
-\alpha_{1}(D_{H}^{n+1},\nabla \times \bchi_{h})+
\alpha_{2}(\nabla \cdot \bw_{h}^{n+1},\nabla \cdot \bchi_{h})=(\nabla\times\bff(t^{n+1}),\bchi_{h}), \label{w1}\\
(\nabla \cdot \bw_{h}^{n+1},\tau_{h})=0, \label{w2}
\\ 
(D_{H}^{n+1}-\nabla \times \bw_{h}^{n},\rho_{H})=0, \label{w3}
\\
\bw_h^{n+1}=I_h(\nabla\times\bu_h^{n+1})\quad \text{on} \quad\partial \Omega, \label{w4}
\\
\bw_h^{n+1}=I_h(\nabla\times\bu_h^{0})\quad \text{for} \quad t=0,\label{w5}
\end{eqnarray}
where $I_h: \nabla\times \bfX_h\longrightarrow \widetilde\bfX_h$ is an interpolation operator and $\lambda_h$
is a multiplier which states for the discrete divergence-free condition for vorticity. Note that due to 
\eqref{w2}, $\bw_{h}$ is also in $\bfX_h$.  

We first state the energy balance of SAV method.
\begin{theorem}
	Solutions of Algorithm \ref{Alg1} satisfy the discrete energy balance:
	\begin{eqnarray}
	\lefteqn{\norm{\bu_h^N}^2+\norm{2\bu_h^N-\bu_h^{N-1}}^2}\nonumber\\
	&&+\Delta t\sum_{n=1}^{N-1}\bigg(\nu\norm {\nabla \bu_h^{n+1}}^2+\alpha_1(\nabla\times(\bu_h^{n+1}-\bu_h^n),\nabla\times\bu_h^{n+1})+\alpha_2\norm{\nabla\cdot\bu_h^{n+1}}^2\bigg)\nonumber\\
&&	=\norm{\bu_h^1}^2+\norm{2\bu_h^1-\bu_h^{0}}^2+(\bff(t^{n+1}),\bu_h^{n+1}).
	\end{eqnarray}
\end{theorem}
\begin{proof}
	Set $\bv_h=\bu_h^{n+1}$ in (\ref{1}) and $q_h=p_h^{n+1}$ in (\ref{2}), and use the identity  $$a(3a-4b+c)=\dfrac{1}{2}((a^2-b^2)+(2a-b)^2-(2b-c)^2+(a-2b+c)^2).$$
	This yields
	\begin{eqnarray}
	\dfrac{1}{4\Delta t}\norm{\bu_h^{n+1}}^2-\dfrac{1}{4\Delta t}\norm{\bu_h^n}^2+\dfrac{1}{2\Delta t}\norm{2\bu_h^{n+1}-\bu_h^n}^2-\dfrac{1}{2\Delta t}\norm{2\bu_h^n-\bu_h^{n-1}}^2\nonumber\\
	+\dfrac{1}{2\Delta t}\norm{\bu_h^{n+1}-2\bu_h^n+\bu_h^{n-1}}^2
		+\alpha_1(\nabla\times(\bu_h^{n+1}-\bu_h^n),\nabla\times\bu_h^{n+1})\nonumber\\+\nu\norm{\nabla\bu_h^{n+1}}^2+\alpha_2\norm{\nabla\cdot \bu_h^{n+1}}^2=(\bff(t^{n+1}),\bu_h^{n+1}).
	\end{eqnarray}
	Now, summing from $n = 1$ to $N-1$ and multiplying each term by $4\Delta t$ proves the stated result.
\end{proof}
We give the helicity balance of the algorithm by using $(\ref{w1})-(\ref{w5})$.
\begin{theorem}
	Solutions of Algorithm \ref{Alg1} satisfy the following discrete helicity balance:
	\begin{eqnarray}
	(\bw_h^N,\bu_h^N)+(2\bw_h^N-\bw_h^{N-1},2\bu_h^N-\bu_h^{N-1})+\sum_{n=1}^{N-1}(\bw_h^{n+1}-2\bw_h^n+\bw_h^{n-1},\bu_h^{n+1}-2\bu_h^n+\bu_h^{n-1})\nonumber\\+\Delta t\sum_{n=1}^{N-1}\bigg(\nu(\nabla\bw_h^{n+1},\nabla\bu_h^{n+1})+\alpha_1(\nabla\times\bw_h^{n+1},\nabla\times(\bu_h^{n+1}-\bu_h^n))+\alpha_2(\nabla\cdot\bw_h^{n+1},\nabla\bu_h^{n+1})\bigg)\nonumber\\=(\bw_h^1,\bu_h^1)+(2\bw_h^1-\bw_h^{0},2\bu_h^1-\bu_h^{0})+2\Delta t\sum_{n=1}^{N-1}\bigg((\bff(t^{n+1}),\bw_h^{n+1})+(\nabla\times\bff(t^{n+1}),\bu_h^{n+1})\bigg).
	\end{eqnarray}
\end{theorem}
\begin{proof}
Choose $\bv_h=\bw_h^{n+1}$ in (\ref{1}) and $\bchi_h=\bu_h^{n+1}$	in (\ref{w1}). Then, the pressure term and one of the nonlinear terms in (\ref{w1}) vanish and we get
\begin{eqnarray}
\bigg(\dfrac{3\bu_{h}^{n+1}-4\bu_{h}^{n}+\bu_{h}^{n-1}}{2\Delta t},\bw_{h}^{n+1}\bigg)+\nu(\nabla \bu_{h}^{n+1},\nabla \bw_{h}^{n+1})+b(2\bu_{h}^{n}-\bu_h^{n-1},\bu_{h}^{n+1},\bw_{h}^{n+1})
\nonumber\\
+\alpha_{1}(\nabla \times \bu_{h}^{n+1},\nabla \times \bw^{n+1}_{h})
-\alpha_{1}(S_{H}^{n+1},\nabla \times \bw_{h}^{n+1})\nonumber\\+
\alpha_{2}(\nabla \cdot \bu_{h}^{n+1},\nabla \cdot \bw_{h}^{n+1})=(\bff(t^{n+1}),\bw_{h}^{n+1})\,\label{uh2} 
\end{eqnarray}
and
\begin{eqnarray}
\bigg(\dfrac{3\bw_{h}^{n+1}-4\bw_{h}^{n}+\bw_{h}^{n-1}}{2\Delta t},\bu_{h}\bigg)+
\nu(\nabla \bw_{h}^{n+1},\nabla \bu_{h}^{n+1})+b(2\bu_{h}^{n}-\bu_h^{n-1},\bw_{h}^{n+1},\bu^{n+1}_{h})
\nonumber\\
+\alpha_{1}(\nabla \times \bw_{h}^{n+1},\nabla \times \bu^{n+1}_{h})
-\alpha_{1}(D_{H}^{n+1},\nabla \times \bu^{n+1}_{h})\nonumber\\+
\alpha_{2}(\nabla \cdot \bw_{h}^{n+1},\nabla \cdot \bu^{n+1}_{h})=(\nabla\times\bff(t^{n+1}),\bu^{n+1}_{h}).\label{wh2}
\end{eqnarray}
Now, setting $l_H=\nabla\times\bw_h^{n+1}$ in (\ref{3}) and $\rho_H=\nabla\times\bu_h^{n+1}$ in (\ref{w3}) we get
\begin{eqnarray}
(S_H^{n+1},\nabla\times\bw_h^{n+1})=(\nabla\times\bu_h^{n},\nabla\times\bw_h^{n+1}), \label{uh3}
\end{eqnarray}
and
\begin{eqnarray}
(D_H^{n+1},\nabla\times\bu_h^{n+1})=(\nabla\times\bw_h^{n},\nabla\times\bu_h^{n+1}).\label{wh3}
\end{eqnarray}
Then, substituting (\ref{uh3}) into the equation (\ref{uh2}) and (\ref{wh3}) into the equation (\ref{wh2}) leads to
\begin{eqnarray}
\bigg(\dfrac{3\bu_{h}^{n+1}-4\bu_{h}^{n}+\bu_{h}^{n-1}}{2\Delta t},\bw_{h}^{n+1}\bigg)+\nu(\nabla \bu_{h}^{n+1},\nabla \bw_{h}^{n+1})+b(2\bu_{h}^{n}-\bu_h^{n-1},\bu_{h}^{n+1},\bw_{h}^{n+1})\nonumber\\
+\alpha_{1}(\nabla \times (\bu_{h}^{n+1}-\bu_h^n),\nabla \times \bw_{h}^{n+1})+
\alpha_{2}(\nabla \cdot \bu_{h}^{n+1},\nabla \cdot \bw_{h}^{n+1})=(\bff(t^{n+1}),\bw_{h}^{n+1}),\label{uh4}
\end{eqnarray}
and
\begin{eqnarray}
\bigg(\dfrac{3\bw_{h}^{n+1}-4\bw_{h}^{n}+\bw_{h}^{n-1}}{2\Delta t},\bu_{h}\bigg)+
\nu(\nabla \bw_{h}^{n+1},\nabla \bu_{h}^{n+1})+b(2\bu_{h}^{n}-\bu_h^{n-1},\bw_{h}^{n+1},\bu^{n+1}_{h})\nonumber\\
+\alpha_{1}(\nabla \times (\bw_{h}^{n+1}-\bw_h^n),\nabla \times \bu^{n+1}_{h})
+
\alpha_{2}(\nabla \cdot \bw_{h}^{n+1},\nabla \cdot \bu^{n+1}_{h})=(\nabla\times\bff(t^{n+1}),\bu^{n+1}_{h}).\label{wh4}
\end{eqnarray}
Next, rewriting the first terms on the left hand sides of (4.13) and (4.10) yields
\begin{eqnarray}
\dfrac{1}{4\Delta t}(\bu_h^{n+1},\bw_h^{n+1})-\dfrac{1}{4\Delta t}(\bu_h^{n},\bw_h^{n})+\dfrac{1}{4\Delta t}(2\bu_h^{n+1}-\bu_h^n,\bw_h^{n+1}-\bw_h^n)-\dfrac{1}{4\Delta t}(2\bu_h^{n}-\bu_h^{n-1},\bw_h^{n}-\bw_h^{n-1})\nonumber\\
+\dfrac{1}{4\Delta t}(\bu_h^{n+1}-2\bu_h^{n}+\bu_h^{n-1},\bw_h^{n+1}-2\bw_h^{n}+\bw_h^{n-1})+\nu(\nabla\bu_h^{n+1},\nabla\bw_h^{n+1})
+b(2\bu_h^n-\bu_h^{n-1},\bu_h^{n+1},\bw_h^{n+1})\nonumber\\+\alpha_{1}(\nabla \times (\bu_{h}^{n+1}-\bu_h^n),\nabla \times \bw^{n+1}_{h})+
\alpha_{2}(\nabla \cdot \bu_{h}^{n+1},\nabla \cdot \bw_{h}^{n+1})=(\bff(t^{n+1}),\bw_{h}^{n+1}),\nonumber
\end{eqnarray}
and
\begin{eqnarray}
\dfrac{1}{4\Delta t}(\bw_h^{n+1},\bu_h^{n+1})-\dfrac{1}{4\Delta t}(\bw_h^{n},\bu_h^{n})+\dfrac{1}{4\Delta t}(2\bw_h^{n+1}-\bw_h^n,\bu_h^{n+1}-\bu_h^n)-\dfrac{1}{4\Delta t}(2\bw_h^{n}-\bw_h^{n-1},\bu_h^{n}-\bu_h^{n-1})\nonumber\\
+\dfrac{1}{4\Delta t}(\bw_h^{n+1}-2\bw_h^{n}+\bw_h^{n-1},\bu_h^{n+1}-2\bu_h^{n}+\bu_h^{n-1})+\nu(\nabla\bw_h^{n+1},\nabla\bu_h^{n+1})+b(2\bu_h^n-\bu_h^{n-1},\bw_h^{n+1},\bu_h^{n+1})\nonumber\\
+\alpha_{1}(\nabla \times (\bw_{h}^{n+1}-\bw_h^n),\nabla \times \bu^{n+1}_{h})+
\alpha_{2}(\nabla \cdot \bw_{h}^{n+1},\nabla \cdot \bu_{h}^{n+1})=(\nabla\times\bff(t^{n+1}),\bu_{h}^{n+1}).\nonumber
\end{eqnarray}
Add these two equations and use (\ref{d2}) to obtain
\begin{eqnarray}
\dfrac{1}{2\Delta t}(\bu_h^{n+1},\bw_h^{n+1})-\dfrac{1}{2\Delta t}(\bu_h^{n},\bw_h^{n})+\dfrac{1}{2\Delta t}(2\bu_h^{n+1}-\bu_h^n,\bw_h^{n+1}-\bw_h^n)-\dfrac{1}{2\Delta t}(2\bu_h^{n}-\bu_h^{n-1},\bw_h^{n}-\bw_h^{n-1})\nonumber\\
+\dfrac{1}{2\Delta t}(\bu_h^{n+1}-2\bu_h^{n}+\bu_h^{n-1},\bw_h^{n+1}-2\bw_h^{n}+\bw_h^{n-1})
+2\alpha_{1}(\nabla \times (\bu_{h}^{n+1}-\bu_h^n),\nabla \times \bw^{n+1}_{h})\nonumber\\+2\nu(\nabla\bu_h^{n+1},\nabla\bw_h^{n+1})
+2\alpha_{2}(\nabla \cdot \bu_{h}^{n+1},\nabla \cdot \bw_{h}^{n+1})=(\bff(t^{n+1}),\bw_{h}^{n+1})+(\nabla\times\bff(t^{n+1}),\bu_{h}^{n+1}).\nonumber
\end{eqnarray}
Finally, summing over time-steps and multiplying both sides by $2\Delta t$ gives the required helicity balance.
\end{proof}
\section{Numerical Analysis} \label{Sec}	
This section provides unconditional stability result and convergence analysis of the proposed Algorithm \ref{Alg1}. To do this, for theoretical analysis we present the finite element discretization in  ${\bf V}_h$.\\ 
Then, BDF2LE based SAV method in ${\bf V}_h$ reads as follows: Find $(\bu_{h}^{n+1}, S_{H}^{n+1})\in ({\bf V}_h,L_H)$  satisfying $\forall (\bv_h,l_H)\in ({\bf V}_h,L_H)$.  
	\begin{eqnarray}
	\bigg(\dfrac{3\bu_{h}^{n+1}-4\bu_{h}^{n}
		+\bu_{h}^{n-1}}{2\Delta t},\bv_{h}\bigg)+\nu(\nabla \bu_{h}^{n+1},\nabla \bv_{h})+b(2\bu_{h}^{n}-\bu_h^{n-1},\bu_{h}^{n+1},\bv_{h})\nonumber\\
	+\alpha_{1}(\nabla \times \bu_{h}^{n+1},\nabla \times \bv_{h})
	-\alpha_{1}(S_{H}^{n+1},\nabla \times \bv_{h})\nonumber\\
	+\alpha_{2}(\nabla \cdot \bu_{h}^{n+1},\nabla \cdot \bv_{h})=(\bff(t^{n+1}),\bv_{h}) \label{vh1}
	\\ 
	(S_{H}^{n+1}-\nabla \times \bu_{h}^{n},l_{H})=0 \label{vh3}
	\end{eqnarray}
The following lemma is required for proving the existence of discrete solutions to \eqref{vh1}-\eqref{vh3}, see \cite{Layton}. This motives the more detailed error analysis that follows.	
\begin{lemma}\label{lemma} Let $ \bff$ $\in L^{\infty}(0,T;H^{-1}(\Omega)) $ and $\bu_h$ be a solution of Algorithm \ref{Alg1}. Then, for any $\Delta t>0$ and $N\geq1$ 
			\begin{eqnarray}
		\lVert {\bu_{h}^{N}} \rVert^{2} +\lVert{2\bu_{h}^{N}-\bu_{h}^{N-1}} \rVert ^{2} +2\Delta t \sum_{n=1}^{N-1}\big(\nu \lVert {\nabla \bu_{h}^{n+1}} \rVert^{2} +2\alpha_{2} \lVert {\nabla \cdot \bu_{h}^{n+1}}\rVert^{2}\big)+2\alpha_1\Delta t\norm{\nabla\times \bu_h^N}^2\nonumber
		\\
		\leq \lVert {\bu_{h}^{1}} \rVert ^{2}+\lVert {2\bu_{h}^{1}-\bu_{h}^{0}} \rVert ^{2}+2\alpha_1\Delta t\norm{\nabla\times \bu_h^1}^2 +2\Delta t \sum_{n=1}^{N-1}  \nu^{-1}\lVert {\bff(t^{n+1})} \rVert ^{2}_{-1}. 
		\end{eqnarray}
	\end{lemma}	
	\begin{proof} To start the proof, we first choose $\bv_{h}=\bu_{h}^{n+1}$ in (\ref{vh1}), vanishing the skew-symmetric trilinear term to obtain 
			\begin{eqnarray}
		\bigg(\dfrac{3\bu_{h}^{n+1}-4\bu_{h}^{n}+\bu_{h}^{n-1}}{2\Delta t},\bu_{h}^{n+1}\bigg)+v(\nabla \bu_{h}^{n+1},\nabla \bu_{h}^{n+1})
		+\alpha_{1}(\nabla \times \bu_{h}^{n+1}\nabla \times \bu_{h}^{n+1})\nonumber\\
		+\alpha_{2}(\nabla \cdot \bu_{h}^{n+1},\nabla \cdot \bu_{h}^{n+1})
		=\alpha_{1}(S_{H}^{n+1},\nabla \times  \bu_{h}^{n+1})+(\bff(t^{n+1}),\bu_{h}^{n+1}).\label{eq1}
		\end{eqnarray}
		Next, for the first term in the left hand side of (\ref{eq1}), we use the identity;
		\begin{eqnarray}
		\dfrac{1}{2}(3a-4b+c)a=\dfrac{1}{4}[a^{2}+(2a-b)^{2}]-\dfrac{1}{4}[b^{2}+(2b-c)^{2}]+\dfrac{1}{4}(a-2b+c)^{2}. \label{id}
		\end{eqnarray}
		Then, one has
			\begin{eqnarray}
		\dfrac{1}{4\Delta t}\bigg[\lVert{\bu_{h}^{n+1}}\rVert^{2}+\lVert{2\bu_{h}^{n+1}-\bu_{h}^{n}}\rVert^{2}\bigg]-\dfrac{1}{4\Delta t}\bigg[\lVert{\bu_{h}^{n}}\rVert^{2}+\lVert{2\bu_{h}^{n}-\bu_{h}^{n-1}}\rVert^{2}\bigg]+\dfrac{1}{4\Delta t}\lVert{\bu_{h}^{n+1}-2\bu_{h}^{n}+\bu_{h}^{n-1}}\rVert^{2}\nonumber
		\\
		+\nu\lVert{\nabla \bu_{h}^{n+1}}\rVert^{2}+\alpha_{1}\lVert{\nabla\times \bu_{h}^{n+1}}\rVert^{2}+\alpha_{2}\lVert{\nabla \cdot \bu_{h}^{n+1}}\rVert^{2}
		=\alpha_{1}(S_{H}^{n+1},\nabla\times \bu_{h}^{n+1})+(\bff(t^{n+1}),\bu_{h}^{n+1}) \label{eq2}
		\end{eqnarray}
		The application of Cauchy-Schwarz inequality, Young's inequality and the dual norm on the right hand side terms of (\ref{eq2}) gives 
		\begin{eqnarray}
		\dfrac{1}{4\Delta t}\bigg[\lVert{\bu_{h}^{n+1}}\rVert^{2}+\lVert{2\bu_{h}^{n+1}-\bu_{h}^{n}}\rVert^{2}\bigg]-\dfrac{1}{4\Delta t}\bigg[\lVert{\bu_{h}^{n}}\rVert^{2}+\lVert{2\bu_{h}^{n}-\bu_{h}^{n-1}}\rVert^{2}\bigg]+\dfrac{1}{4\Delta t}\lVert{\bu_{h}^{n+1}-2\bu_{h}^{n}+\bu_{h}^{n-1}}\rVert^{2}\nonumber
		\\
		+\dfrac{\nu}{2}\lVert{\nabla \bu_{h}^{n+1}}\rVert^{2}+\dfrac{\alpha_{1}}{2}\lVert{\nabla\times \bu_{h}^{n+1}}\rVert^{2}+\alpha_{2}\lVert{\nabla \cdot \bu_{h}^{n+1}}\rVert^{2}
		\leq \dfrac{\alpha_{1}}{2}\norm{S_{H}^{n+1}}^2+\dfrac{1}{2\nu}\|\bff(t^{n+1})\|_{-1}^2 \label{eq3}
		\end{eqnarray}
		Note that choosing $l_{H}=S_{H}^{n+1}$ in (\ref{vh3}) and using the Cauchy-Schwarz and Young's inequalities, we obtain the following bound for the first term on the right hand side of (\ref{eq3}) 
		\begin{eqnarray}
		\lVert{S_{H}^{n+1}}\rVert^{2}=(S_{H}^{n+1},\nabla\times \bu_{h}^n)\leq \lVert{\nabla\times \bu_{h}^n}\rVert \lVert {S_{H}^{n+1}}\rVert\leq \dfrac{1}{2} \lVert {\nabla\times \bu_{h}^n}\rVert^{2}+\dfrac{1}{2}\lVert {S_{H}^{n+1}}\rVert^{2},
		\end{eqnarray}
		so that
			\begin{eqnarray}\lVert{S_{H}^{n+1}}\rVert\leq\lVert{\nabla\times \bu_{h}^n}\rVert \label{5}.
		\end{eqnarray}
Using the above inequality, dropping the non-negative term $\dfrac{1}{4\Delta t}\lVert{\bu_{h}^{n+1}-2\bu_{h}^{n}+\bu_{h}^{n-1}}\rVert^{2}$ and rearranging terms in (\ref{eq3}), we obtain
		\begin{eqnarray}
		\dfrac{1}{4\Delta t}\bigg[\lVert{\bu_{h}^{n+1}}\rVert^{2}-\lVert{\bu_{h}^{n}}\rVert^{2}\bigg]+\dfrac{1}{4\Delta t}\bigg[\lVert{2\bu_{h}^{n+1}-\bu_{h}^{n}}\rVert^{2}-\lVert{2\bu_{h}^{n}-\bu_{h}^{n-1}}\rVert^{2}\bigg]
		+\dfrac{\nu}{2}\lVert{\nabla \bu_{h}^{n+1}}\rVert^{2}\nonumber\\
		+\dfrac{\alpha_{1}}{2}\lVert{\nabla\times \bu_{h}^{n+1}}\rVert^{2}+\alpha_{2}\lVert{\nabla \cdot \bu_{h}^{n+1}}\rVert^{2}
		\leq \dfrac{\alpha_{1}}{2}\lVert{\nabla\times \bu_{h}^n}\rVert^2+\dfrac{\nu^{-1}}{2}\|\bff(t^{n+1})\|_{-1}^2. \label{eq4}
		\end{eqnarray}
	Multiplying both side of the inequality by $4\Delta t$ and taking sum from $n=1$ to $n=N-1$ yields the required stability bound.
	\end{proof}

 We proceed to present an error analysis of our method. The following notations are used for the discrete norms. For $\bv^n \in H^p(\Omega)$, we define
	\begin{equation*}
	\norm{|\bv|}_{\infty,p}:=\max_{0\leq n\leq N}\|\bv^n\|_p ,  \quad \norm{|\bv|}_{m,p}:=\bigg(\Delta t \sum_{n=0}^{N} \|\bv^n\|_p^m  \bigg)^{\frac{1}{m}} \label{dn}.
	\end{equation*}
	To obtain the optimal error estimations, we assume that the following regularity assumptions are satisfied by the analytical solution:
	\begin{equation} \label{ra}
	\begin{aligned}
	&\bu \in L^{\infty}(0,T;H^1(\Omega))\cap H^1(0,T;H^{k+1}(\Omega))\cap H^3(0,T;L^2(\Omega))\cap H^2(0,T;H^1(\Omega)),\\
	&p \in L^2(0,T;H^{k}(\Omega))\quad \text{and} \quad f \in L^2(0,T;L^2(\Omega)).
	\end{aligned}
	\end{equation}
		\begin{theorem}\label{error}
Let $(\bu,p)$ be the solution of the NSE such that the regularity assumptions (\ref{ra}) are satisfied. Then, for any $ N$, the following bound holds for the difference $\be^n=\bu^n-\bu_h^n$:
		\begin{eqnarray}
		\lefteqn{\|\be^{N}\|^2+ \|2\be^{N}-\be^{N-1}\|^2+\sum_{n=1}^{N-1}\|\be^{n+1}-2\be^n+\be^{n-1}\|^2+\nu\Delta t\sum_{n=1}^{N-1}\norm{\nabla\be^{n+1}}^2}\nonumber\\
		&\leq&\exp(C\nu^{-1}T)\bigg[\|\be^1\|^2+\|2\be^1-\be^{0}\|^2+\nu^{-1}h^{2k+2}\norm{|\bu_t|}^2_{2,k+1}+\nu^{-1}h^{2k}\norm{|\bu|}^2_{2,k+1}\norm{|\nabla \bu|}^2_{\infty,0}\nonumber\\
		&&+\nu h^{2k}\norm{|\bu|}^2_{2,k+1}+\alpha_2^{-1}h^{2k+2}\norm{|p|}^2_{2,k+1}+\nu^{-1}\alpha_1^2h^{2k}\norm{|\bu|}^2_{2,k+1}+\nu^{-1}\alpha_1^2H^{2k}\norm{|\bu|}^2_{2,k+1}\nonumber\\
		&&+\nu^{-1}\alpha_1^2\Delta t^2\norm{|\bu_t|}^2_{\infty,0}+\nu^{-1}\Delta t^4\norm{|\bu_{ttt}|}^2_{2,0}+\nu^{-1}\Delta t^{{4}}\norm{|\nabla\bu|}^2_{\infty,0}\norm{|\nabla\bu_{tt}|}^2\bigg]
		\label{ert}
		\end{eqnarray}
	\end{theorem} 
	
	\begin{proof}
	See Appendix A.
	\end{proof}
	
Theorem \ref{error}, approximation properties (\ref{ap1})-(\ref{ap10}) and the estimation (\ref{pa})  immediately yields the following Corollary, proving second order accuracy both in time and space.
	\begin{corollary} \label{cor}
In addition to the regularity assumptions of (\ref{ra}), consider the Taylor-Hood finite element spaces $(\bfX_h,Q_h)=(P_2,P_1)$, the coarse mesh size $H\leq {O}({h}^{1/2})$ , the artificial viscosity parameter $\alpha_1=h^2$ and the grad-div stabilization parameter $\alpha_2\leq O(1)$. Then, assuming $\be^0\neq\be^1$ to be $0$, the error in velocity satisfies, for all $\Delta t>0$
		\begin{eqnarray}
		\|\be^N\|^2 + \|2\be^{N}-\be^{N-1}\|^2+\sum_{n=1}^{N-1}\|\be^{n+1}-2\be^n+\be^{n-1}\|^2+\nu\Delta t\sum_{n=1}^{N-1}\norm{\nabla\be^{n+1}}^2
		\leq C \bigg(h^4+\Delta t^4\bigg).
		\end{eqnarray}
	\end{corollary}

\section{Numerical Experiments}

In this part, we provide three numerical illustrations to test the numerical algorithm (\ref{1})-(\ref{3}). The first test verifies the order of the convergence rates which are predicted Corollary \ref{cor}. In addition, we demonstrate the efficiency of the BDF2LE based SAV method on flow around a cylinder and two dimensional flow between two offset circles problems. All simulations are carried out by considering Taylor-Hood finite element spaces $(P_2,P_1)$ to approximate velocity and pressure and $P_1$ for the large scale space $L_H$. All the numerical experiments are implemented with the finite element software package Freefem++ \cite{hec}.

\subsection{Convergence Rates}
This subsection verifies the predicted convergence rates of our numerical scheme (\ref{1})-(\ref{3}). For this purpose, we consider (\ref{nse}) with the prescribed solution:
		\begin{align}\label{truesol}
		\bu=\left[%
		\begin{array}{c}
		(1+0.01t)sin(2\pi y) \\
		(1+0.01t)cos(2\pi x)%
		\end{array}%
		\right],\quad \quad \quad \quad 
		p=x+y
		\end{align}	
Simulations are performed in a unit square $\Omega:=[0,1]^2$ with $\nu=1$ and the last time $T=0.01$. The coarse mesh size $H=\sqrt{h}$, the parameters $\alpha_1=h^2$ and $\alpha_2=1$ are chosen. The external force $\bf f$ is determined by the true solution (\ref{truesol}). Boundary conditions are set to be true solutions on $\partial \Omega$. We compute approximate solutions on successive mesh refinements and the velocity errors are computed in the discrete norm  $L^2(0,T;{H}^1(\Omega))$
		$$\|\textbf{u}-\textbf{u}^h\|_{2,1}=\left\{\Delta t \sum_{n=1}^{N}\|\textbf{u}(t^n)-\textbf{u}_{n}^{h}\|^{2}\right\}^{1/2}.$$ \\
Results for errors are shown in Table $\ref{table:tab1}$ and second order accuracy is observed, exactly as the theory predicts.
		\begin{table}[h!]
			\begin{center}
				\begin{tabular}{|c|c|c|c|}
					\hline
					$h$ & $\Delta t$ &$\|\textbf{u}-\textbf{u}_h\|_{2,1}$ & Rate\\
					\hline
					1/4&0.01&5.25977e-1 &--  \\
					\hline
					1/8&0.005& 1.3403e-1 &1.96  \\
					\hline
					1/16&0.0025& 3.397e-2 &1.98   \\
					\hline
					1/32&0.00125&8.50843e-3 &1.99  \\
					\hline
					1/64&0.000625&2.13204e-3 &1.99 \\
					\hline
					1/128&0.0003125&5.32992e-4&2.00\\
					\hline
				\end{tabular}
				\caption{Errors and convergence rates for the Algorithm \ref{Alg1}. }
				\label{table:tab1}
			\end{center}
		\end{table}
		\subsection{Flow Around a Cylinder}
The second example is considered to verify the efficiency of our scheme (\ref{1})-(\ref{3}) on two dimensional flow along a rectangular channel in which a cylinder is depicted seen in Figure \ref{fig:d}. This famous benchmark problem is highly preferable since it exhibits real flow characteristics and provides highly reliable data that allows to measure the accuracy of codes. In addition, simulating this flow accurately is critical to observe the behavior of eddies. The study \cite{MS} has addressed this problem and presented the computational results to define the reference values. The accuracy of these reference values has been significantly improved in \cite{J04}.

\begin{figure}[!hb]
	\centering
	\includegraphics[width=110mm]{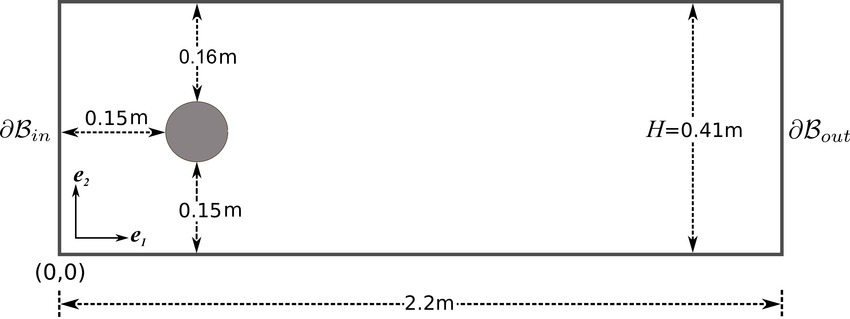}
	\caption{ Domain $\Omega$ of the test problem}
	\label{fig:d}
\end{figure}
 The inflow and outflow velocities are presented as
		\begin{eqnarray}
		u_1(0,y,t)&=&u_1(2.2,y,t)=\dfrac{6}{0.41^2}sin\big(\dfrac{\pi t}{8}\big)y(0.41-y),\nonumber\\\quad\nonumber\\ u_2(0,y,t)&=& u_2(2.2,y,t)= 0.\nonumber
		\end{eqnarray}
We enforce no-slip boundary conditions at the cylinder and walls. We choose zero initial condition $\bu(x,y,t)=0 $, the kinematic viscosity $\nu=10^{-3}$ and the forcing $\bff=0$. Also, we choose artificial viscosity parameters as $\alpha_1=h^2$ and $\alpha_2=0.001$ for regular mesh size $h$ and the coarse mesh size $H=\sqrt{h}$. In all computations, we use a very coarse mesh consisting only $15.485$ total degrees of freedom with a last time $T=8$ and time-step $\Delta t=0.01$.\\
We present the flow development in Figure \ref{fig:u} which matches with the results of \cite{J04, MS}. With increasing inflow, we observe the appearance of two vortices behind the cylinder, see $t=2$ and $t=4$. Then vortices leave the cylinder and the formation of a vortex street is clearly seen, which lasts until $t=8$.

\begin{figure}[!h]
	\centering
	\begin{subfigure}[b]{1\linewidth}
		\includegraphics[width=185 mm]{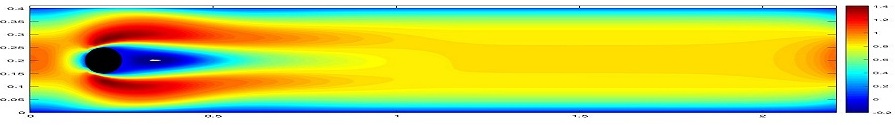}
	\end{subfigure}
	\begin{subfigure}[b]{1\linewidth}
		\includegraphics[width=185 mm]{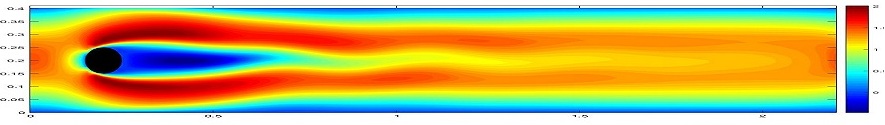}
	\end{subfigure}
	\begin{subfigure}[b]{1\linewidth}
		\includegraphics[width=185 mm]{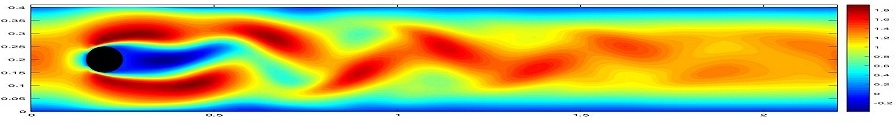}
	\end{subfigure}
	\begin{subfigure}[b]{1\linewidth}
		\includegraphics[width=185 mm]{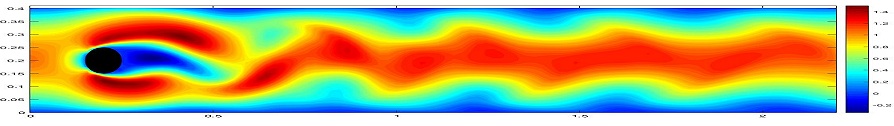}
	\end{subfigure}
	\begin{subfigure}[b]{1\linewidth}
		\includegraphics[width=185 mm]{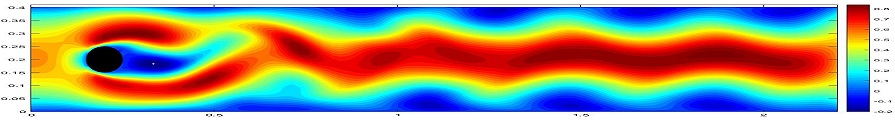}
	\end{subfigure}
	\begin{subfigure}[b]{1\linewidth}
		\includegraphics[width=185 mm]{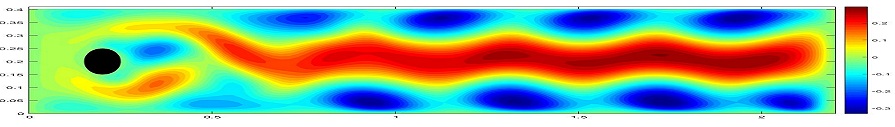}
	\end{subfigure}
	\caption{The velocity of the scheme (\ref{1})-(\ref{3}) at $t = 2, 4, 5, 6, 7, 8$ (from up to down).}
	\label{fig:u}
\end{figure}
\vspace{25 mm}
The most frequently monitored quantities of interest that are considered in the literature for this flow are  the drag $c_d(t)$, the lift coefficient $c_l(t)$ and pressure drop across the object $\Delta p(t)$.  
These values are defined in \cite{MS} as follows:
\begin{eqnarray}
c_d(t)&=&\dfrac{2}{\rho L U_{max}^2}\int_{S}\big(\rho\nu\dfrac{\partial\bu_{t_S}}{\partial n}n_y-p(t)n_x\big)dS\nonumber\\\quad\nonumber\\
c_l(t)&=&-\dfrac{2}{\rho L U_{max}^2}\int_{S}\big(\rho\nu\dfrac{\partial\bu_{t_S}}{\partial n}n_x+p(t)n_y\big)dS\nonumber\\\quad\nonumber\\
\Delta p(t)&=& p(t;0.15,0.2)-p(t;0.25,0.2)\nonumber
\end{eqnarray}
where $S$ is the boundary of the cylinder, $U_{max}$ is the maximum mean flow, $L$ is the diameter of the cylinder, $n=(n_x,n_y)^T$ is the normal vector on the circular boundary $S$ and $\bu_{t_S}$ is the tangential velocity for $t_S=(n_y,-n_x)^T$ the tangential vector.

The plots of evolution of the drag and lift coefficient and the pressure difference are also presented in Figure \ref{fig:dlp} and the graphs are consistent with DNS results of \cite{J04, MS}.
\begin{figure}[!ht]
	\centering
	\begin{subfigure}[b]{0.5\linewidth}
		\includegraphics[width=70 mm]{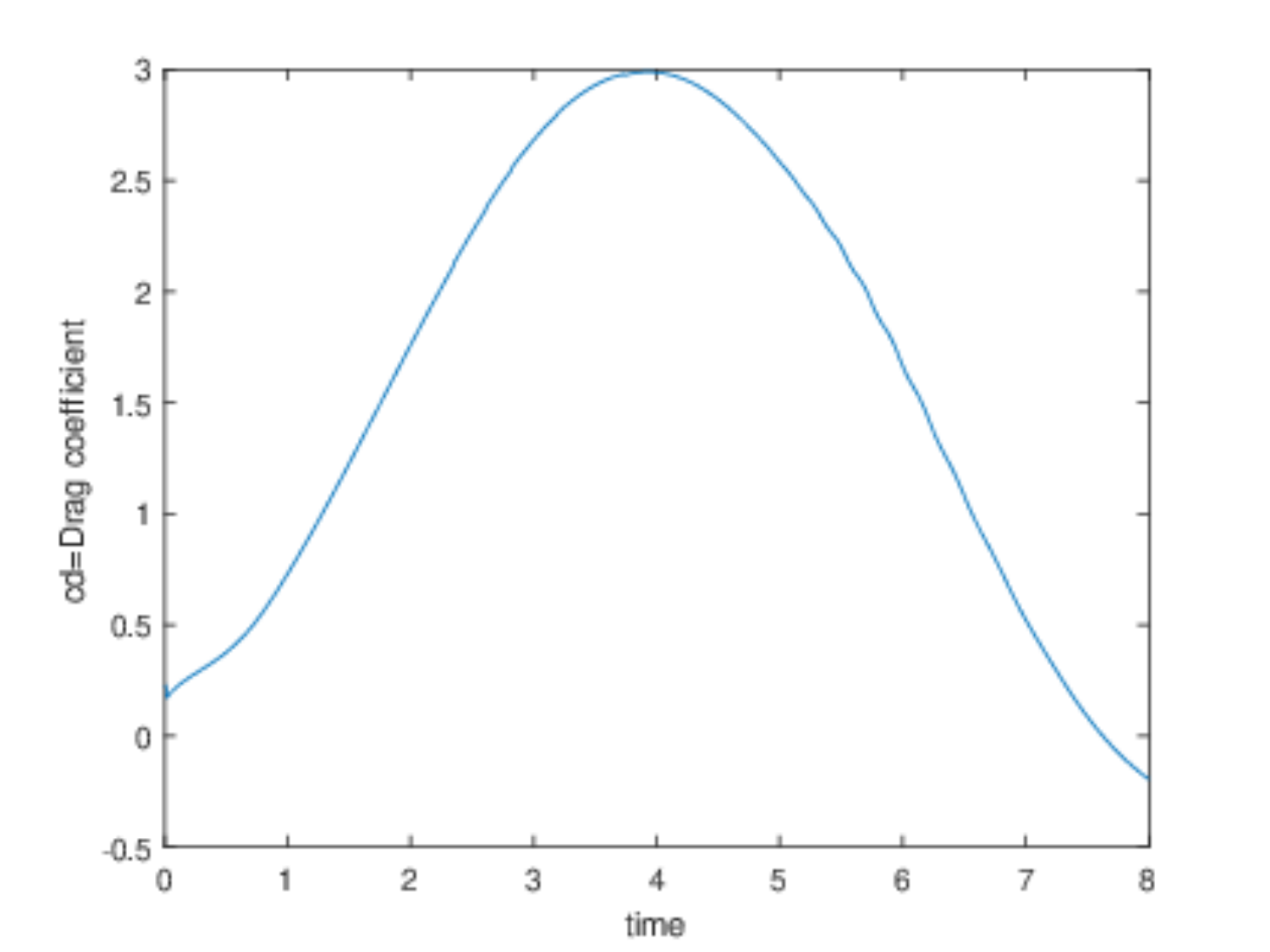}
		\caption{Evolution of $c_{d,max}$}
	\end{subfigure}
	\begin{subfigure}[b]{0.4\linewidth}
		\includegraphics[width=70 mm]{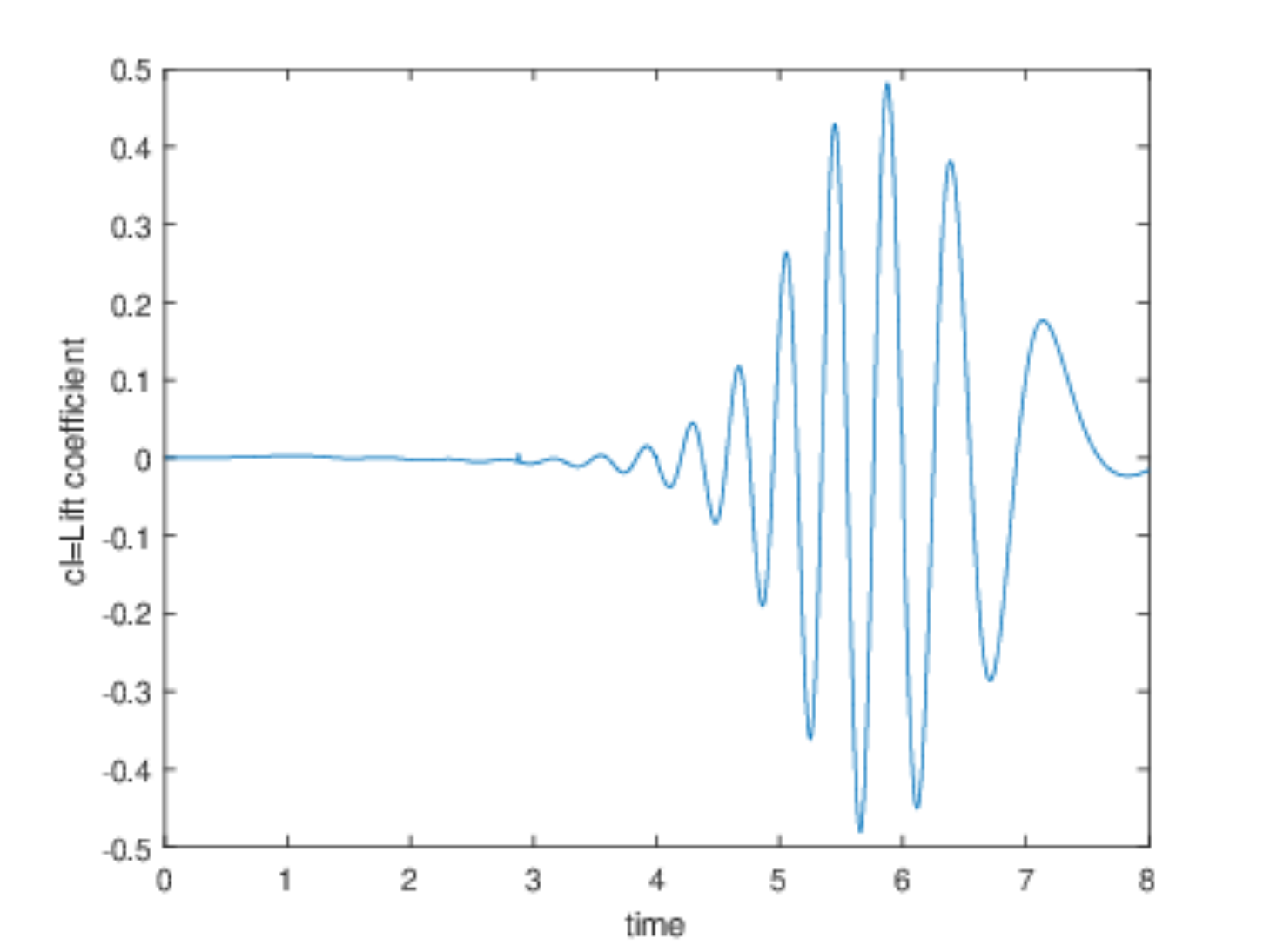}
		\caption{ Evolution of $c_{l,max}$}
	\end{subfigure}
\begin{subfigure}[b]{0.4\linewidth}
	\includegraphics[width=70 mm]{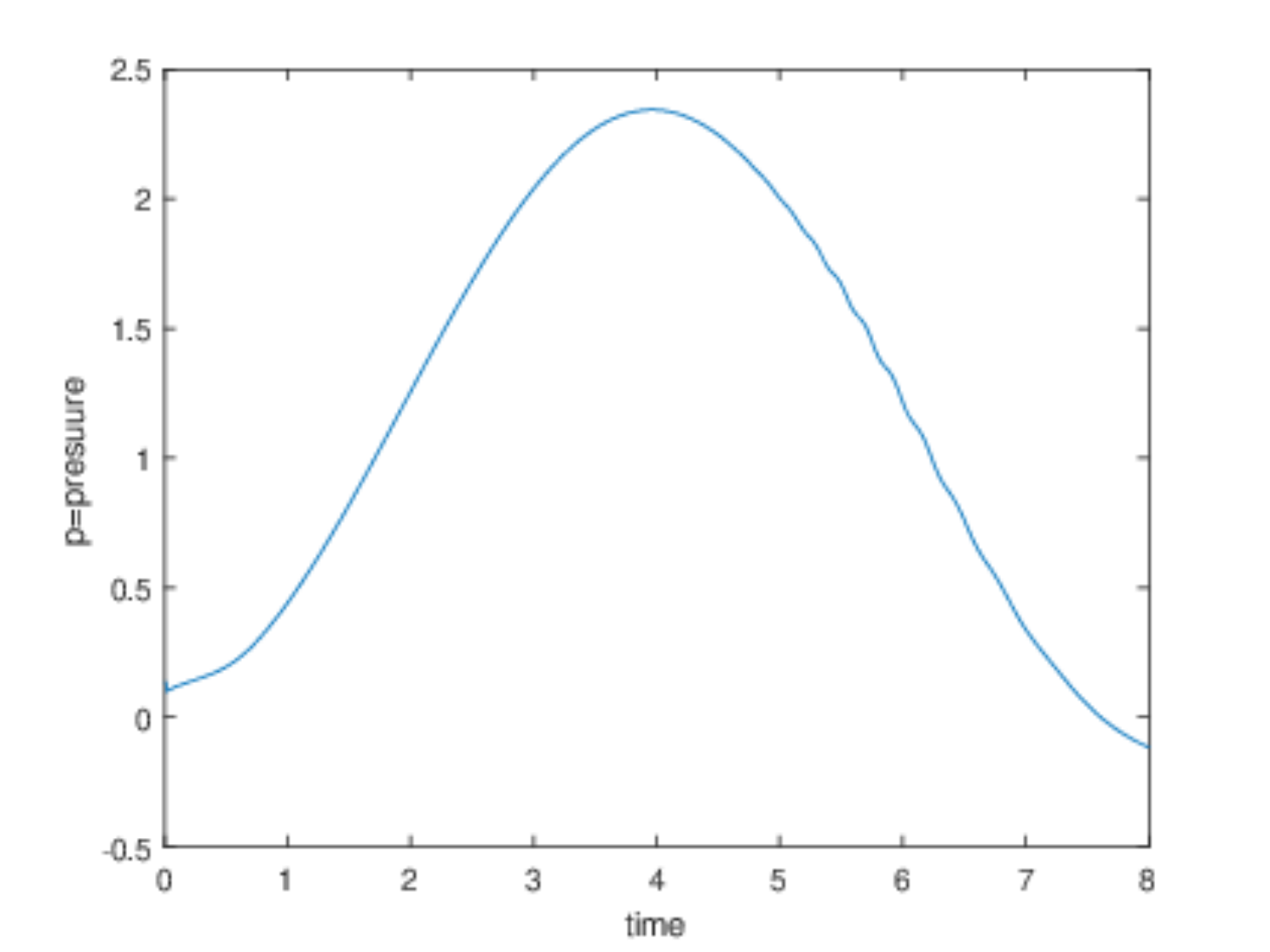}
	\caption{ Evolution of $\Delta p$ }
\end{subfigure}
	\caption{Evolution of maximum value of drag values, lift values and pressure differences obtained when using the scheme (\ref{1})-(\ref{3}) with $\Delta t = 0.01$}
	\label{fig:dlp}
\end{figure}

Note that in Table \ref{table:tab2}, we take only the maximum drag $c_{d,max} $ and maximum lift $c_{l,max}$ values behind the cylinder together with the times at which they occur. We consider the following reference intervals in \cite{MS}:
 $$c_{d,max}^{ref}\in [2.93, 2.97],\quad \quad c_{l,max}^{ref}\in [0.47, 0.49].$$
From this we observe that SAV with BDF2LE method provides the most accurate results for $c_{d,max} $ and $c_{l,max} $ compared with \cite{G} in which SAV method is studied by using CN time stepping scheme.
\begin{table}[h!]
	\begin{center}
\begin{tabular}{|c|c|c|c|c|}
			\hline
			Method & $c_{d,max}$ &$t(c_{d,max})$ & $c_{l,max}$&$t(c_{l,max})$ \\
			\hline
			Ref \cite{G}&2.87198&3.685 &0.436564&5.784  \\
			\hline
			Ref \cite{J04} &2.95092&3.93 &0.477995&5.69\\
			\hline
			Proposed Method&2.98803 & 3.93 &0.480608&5.88 \\
			\hline
		\end{tabular}
	\end{center}
	\caption{Comparison of maximum drag and lift coefficients and the times at which they occur.}
\label{table:tab2}
\end{table}

We like to note that the numerical studies of \cite{MS} are performed on very fine mesh consisting of $785000$ total degrees of freedom. Herein, we can find good results with a less computational time than required by a DNS. 
\section{ Flow Between Two Offset Circles}

The last experiment demonstrates the stability of SAV with BDF2LE method on two dimensional flow in annular region between two offset circles. The domain we use is a circle with an interior decentralized circle inside. Pick $r_1=1$, $r_2=0.1$ and $c=(c_1,c_2)=\big(\frac{1}{2},0\big)$. The domain is then given by
$$\Omega=\left\{(x,y): x^2+y^2\leq r_1^2\right\} \cap \left\{(x,y): (x-c_1)^2+(y-c_2)^2\geq r_1^2\right\}$$
The numerical solutions are computed on a Delunay-generated triangular mesh and an example mesh can be seen in Figure \ref{fig:mesh}.
\begin{figure}[h]
	\centering
	\includegraphics[width=90mm]{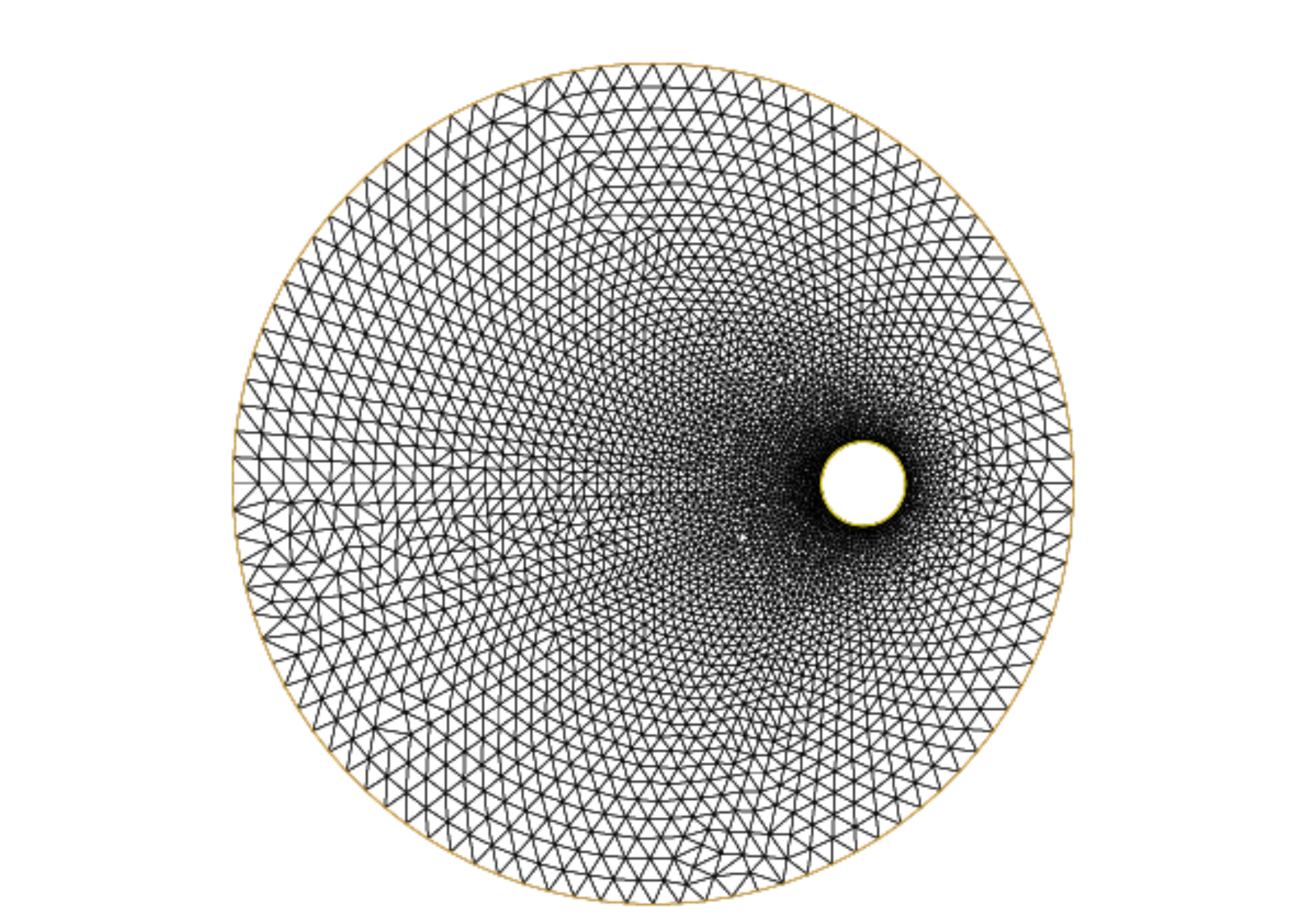}
	\caption{The domain of the test problem}
	\label{fig:mesh}
\end{figure}
Zero initial conditions have been used and no-slip boundary conditions are considered on both circles. We choose time step size $ \Delta t = 0.025$ and
the last time $T = 5$. The flow is generated by the body force rotating in the  counterclockwise direction
$$f(x,y,t)= (-4y(1-x^2-y^2), 4x(1-x^2-y^2))^T$$

Under the influence of the body force, the flow interacts with the inside disk . Then, we observe the formation of a vortex street called von K\'{a}rm\'{a}n. Then, this vortex street moves in the same way creating more complex flow
patterns. Further studies on this flow can be found in \cite{NL,J,LJ,MSL}. To verify the accuracy of our
method, we present some snapshots in Figure \ref{fig:streamline} for $Re=200$ at times $t=0.025, 1.25,3$ and $t = 5$. Snapshots are shown in the qualitative study \cite{NL}, we observe exactly such plots. 

\vspace{20 mm}

\begin{figure}[h!]
	\centering
	\begin{subfigure}[b]{0.4\linewidth}
	\includegraphics[width=80 mm]{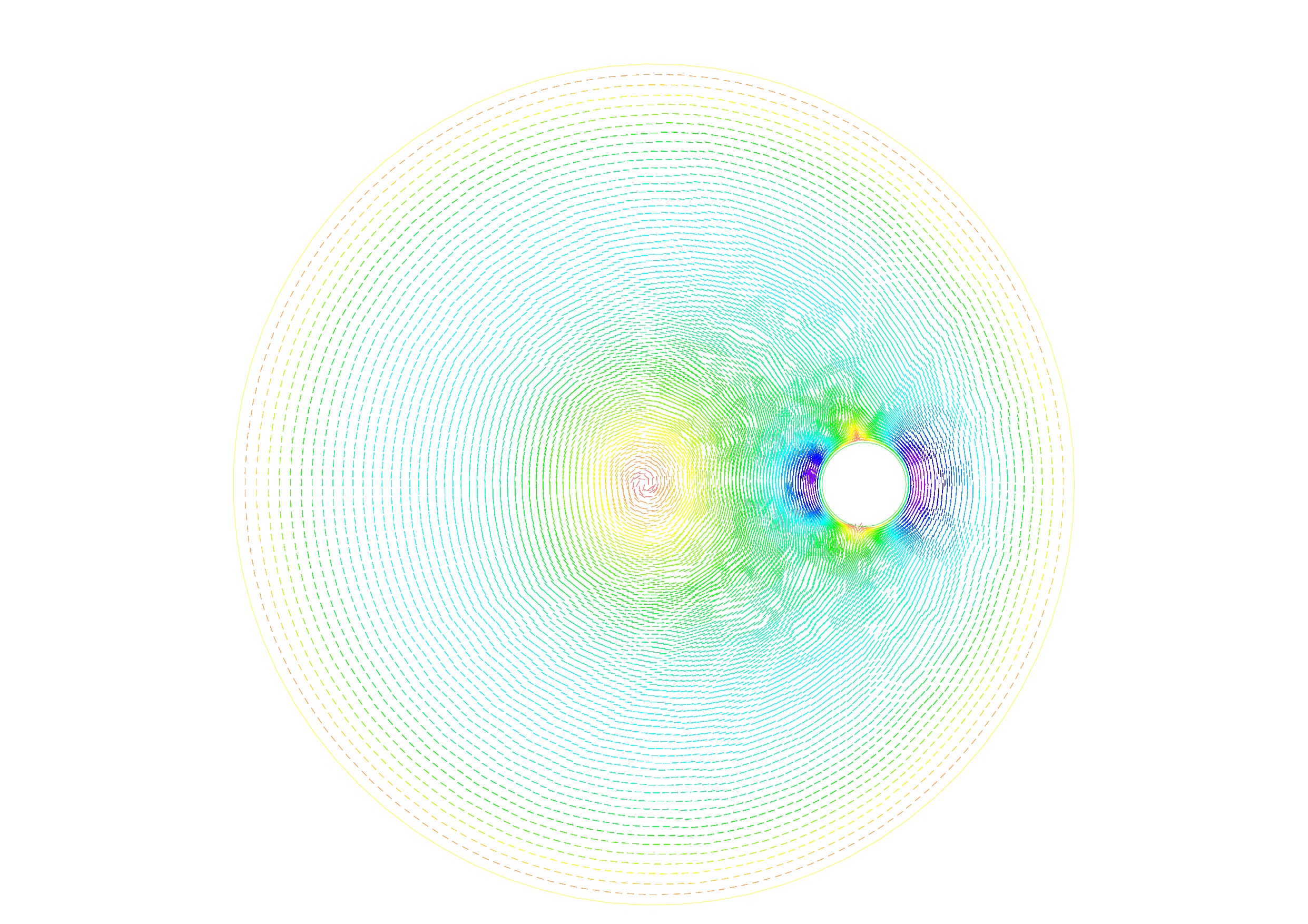}
	\caption{ $t=0.025$}
		\end{subfigure}
	\begin{subfigure}[b]{0.4\linewidth}
		\includegraphics[width=80 mm]{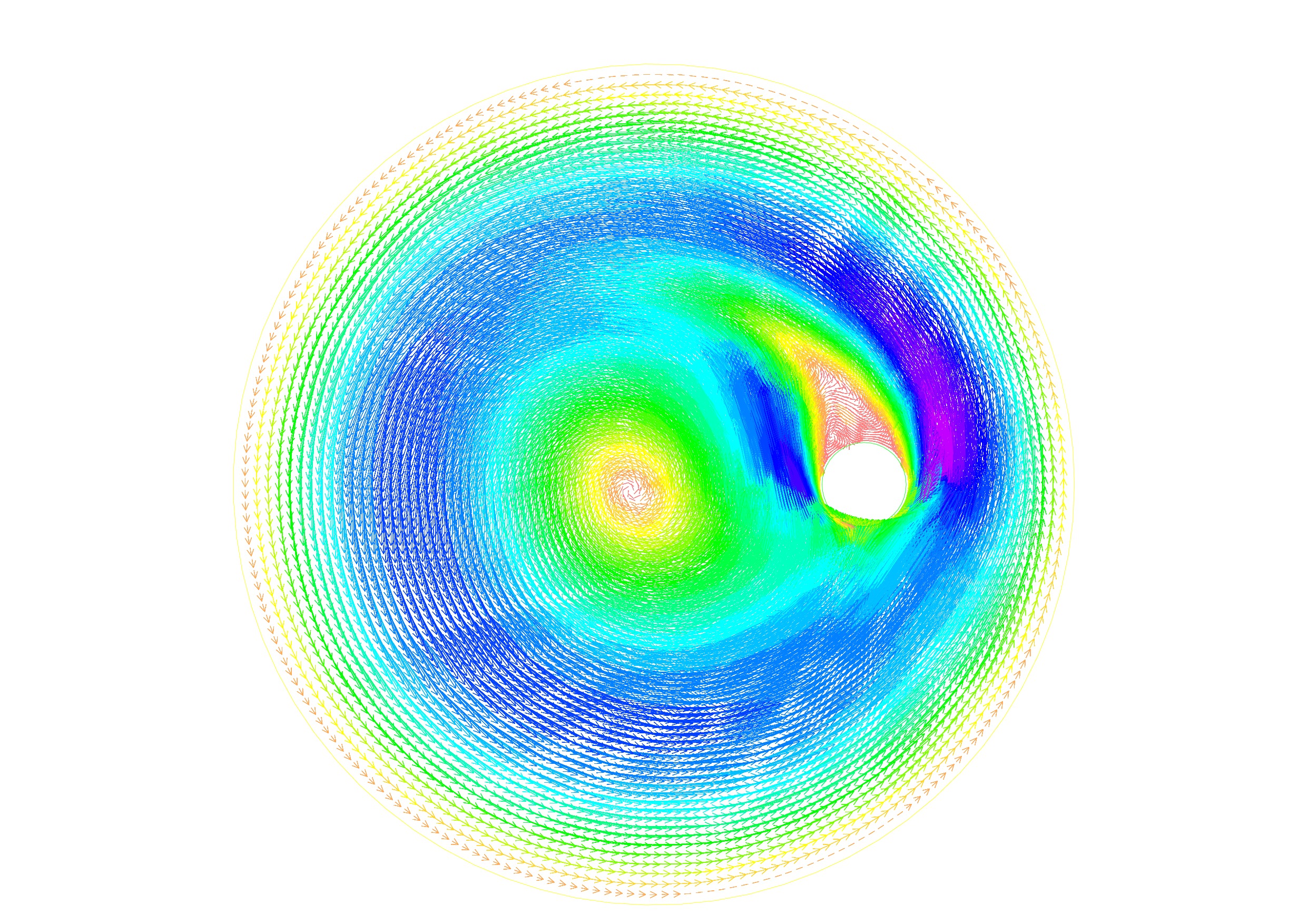}
		\caption{ $t=1.25$}
	\end{subfigure}
\begin{subfigure}[b]{0.4\linewidth}
	\includegraphics[width=80 mm]{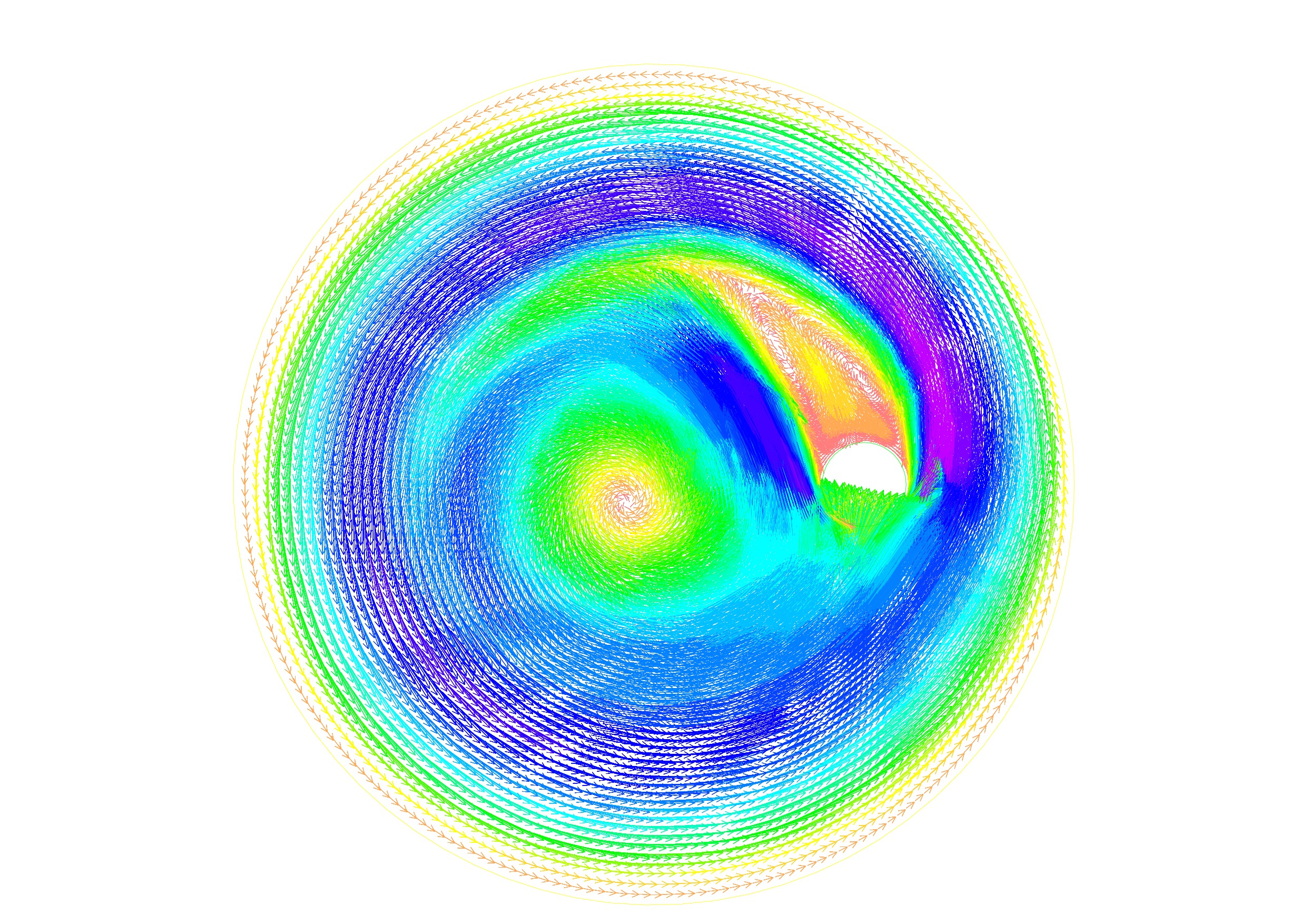}
	\caption{ $t=3$}
\end{subfigure}
\begin{subfigure}[b]{0.4\linewidth}
	\includegraphics[width=80 mm]{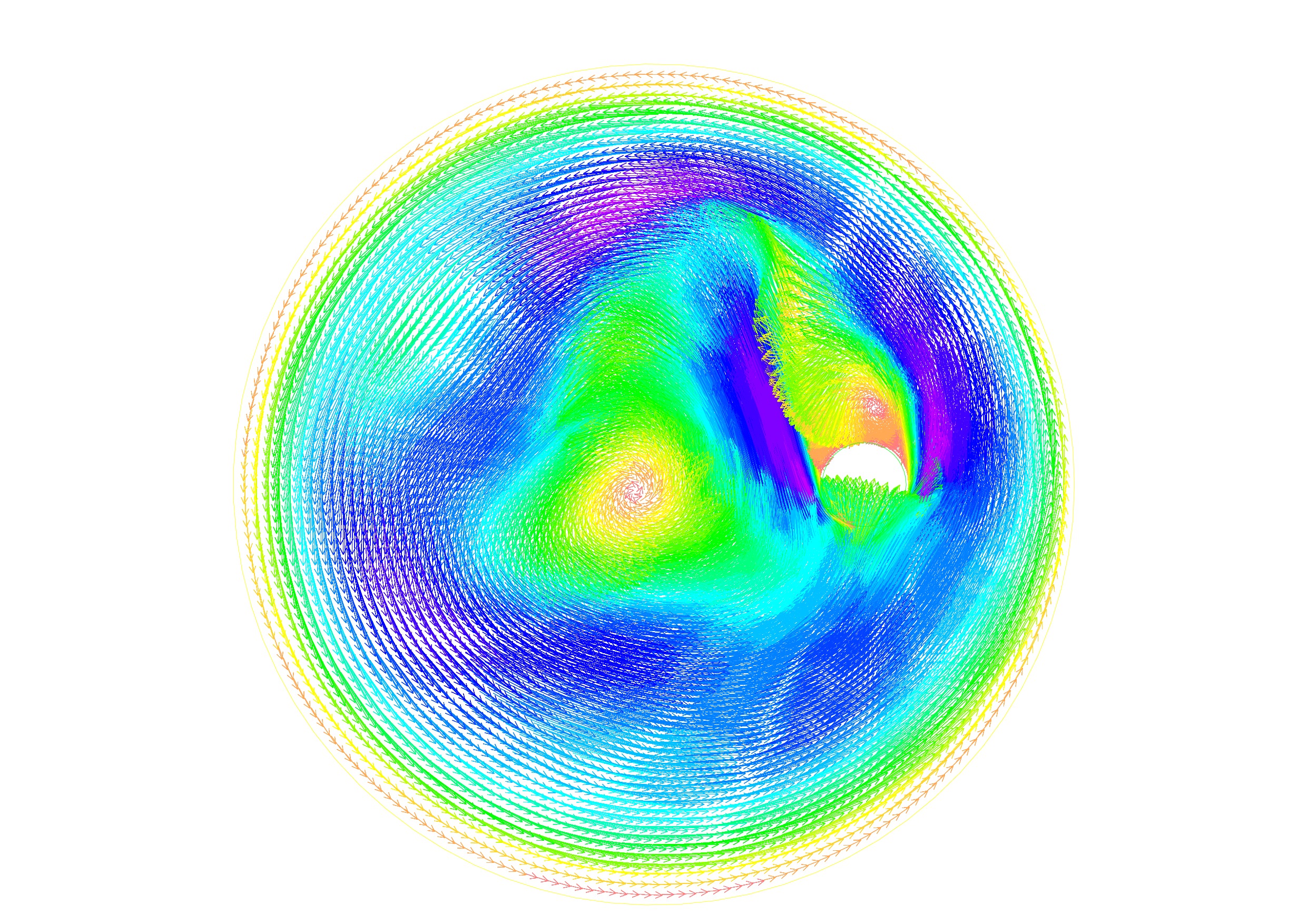}
	\caption{ $t=5$}
\end{subfigure}
\caption{Velocities for $\nu=\frac{1}{200}$}
	\label{fig:streamline}
\end{figure}

Popular quantities of interest in this experiment are the kinetic energy and the enstrophy values for $0 \leq t \leq 5$, defined by

$$Energy=\dfrac{1}{2}\norm{\bu}^2, \quad \quad \quad Enstrophy=\dfrac{1}{2}\nu\norm {\nabla\times\bu}^2$$ 

Figure \ref{fig:Energies} and Figure \ref{fig:Enstrophies} show the kinetic energy and enstrophy statistics for different Reynolds numbers. The curve marked with `BDF2LE-NOSAV' is computed by using BDF2LE method without SAV and `BDF2LE-SAV' by using BDF2LE method with SAV. For $Re=200$, we see that both methods are stable and consistent with \cite{MSL}. However, as Reynolds number increases, solution to BDF2LE-NOSAV method oscillates while BDF2LE-SAV method stays stable as seen in Figure \ref{fig:Energies} and Figure \ref{fig:Enstrophies}.
	\begin{figure}[h!]
	\centering
	\begin{subfigure}[b]{0.3\linewidth}
	\includegraphics[width=55 mm]{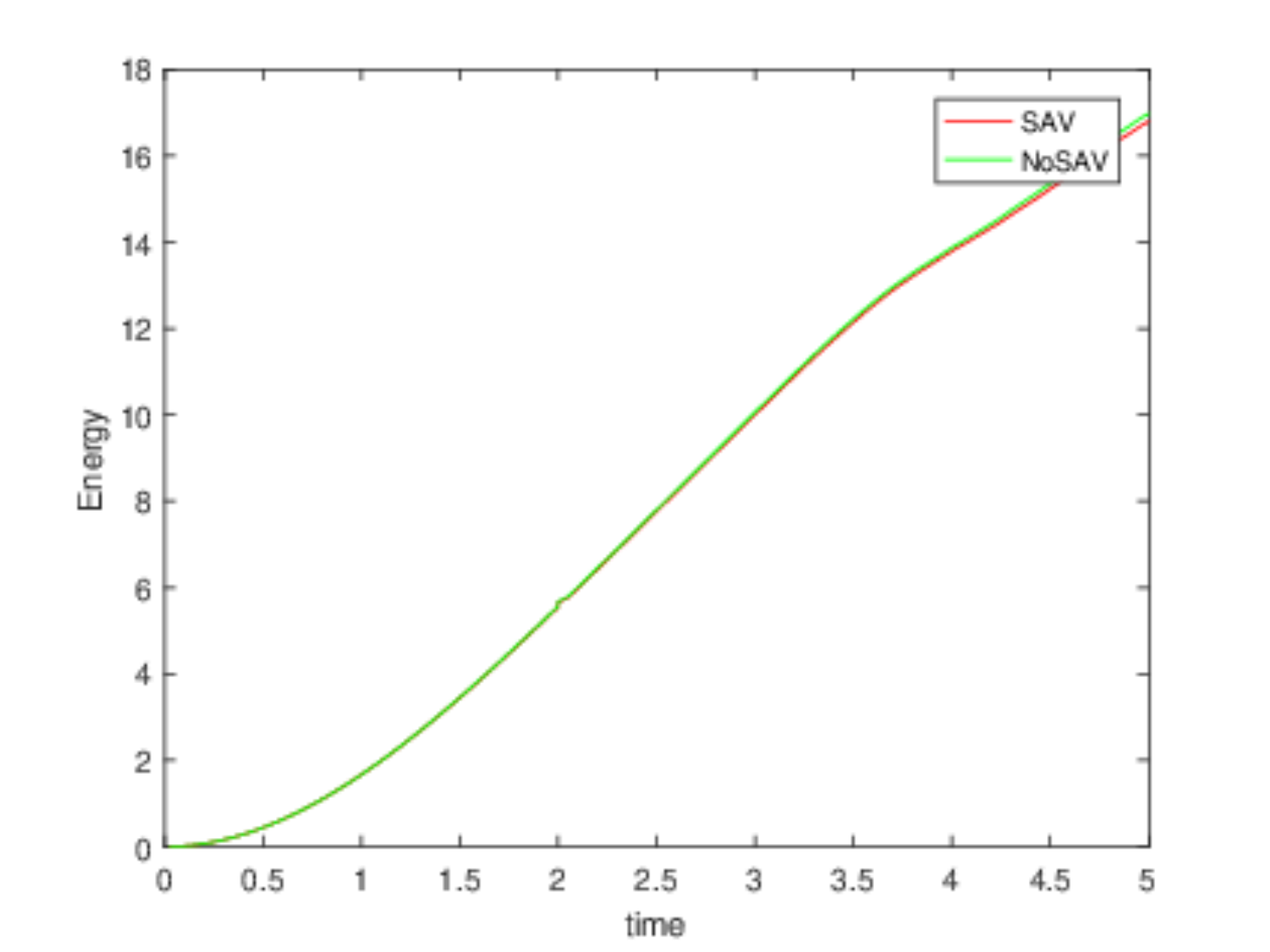}
	\end{subfigure}
\begin{subfigure}[b]{0.3\linewidth}
	\includegraphics[width=55 mm]{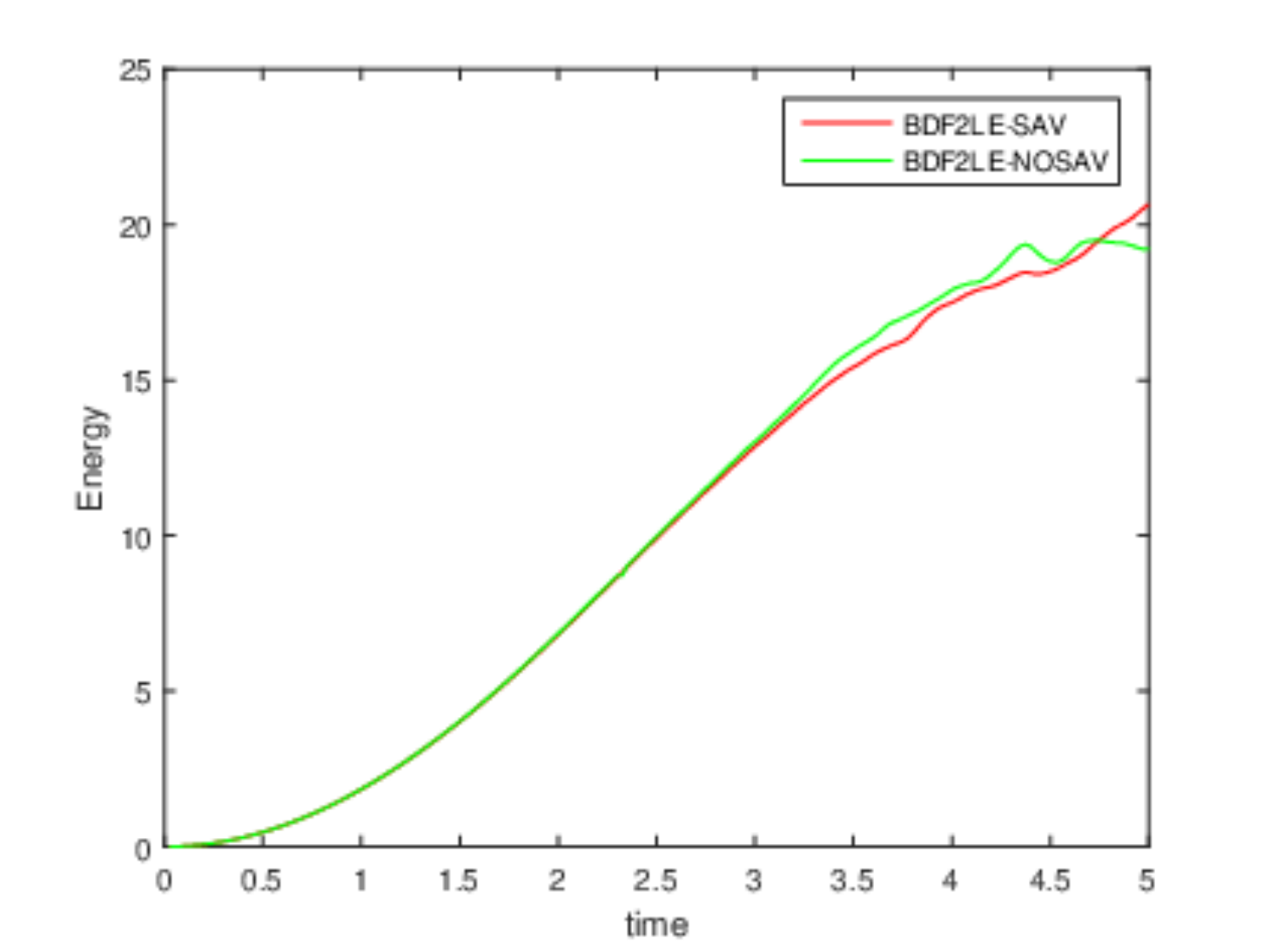}
\end{subfigure}
\begin{subfigure}[b]{0.3\linewidth}
	\includegraphics[width=55 mm]{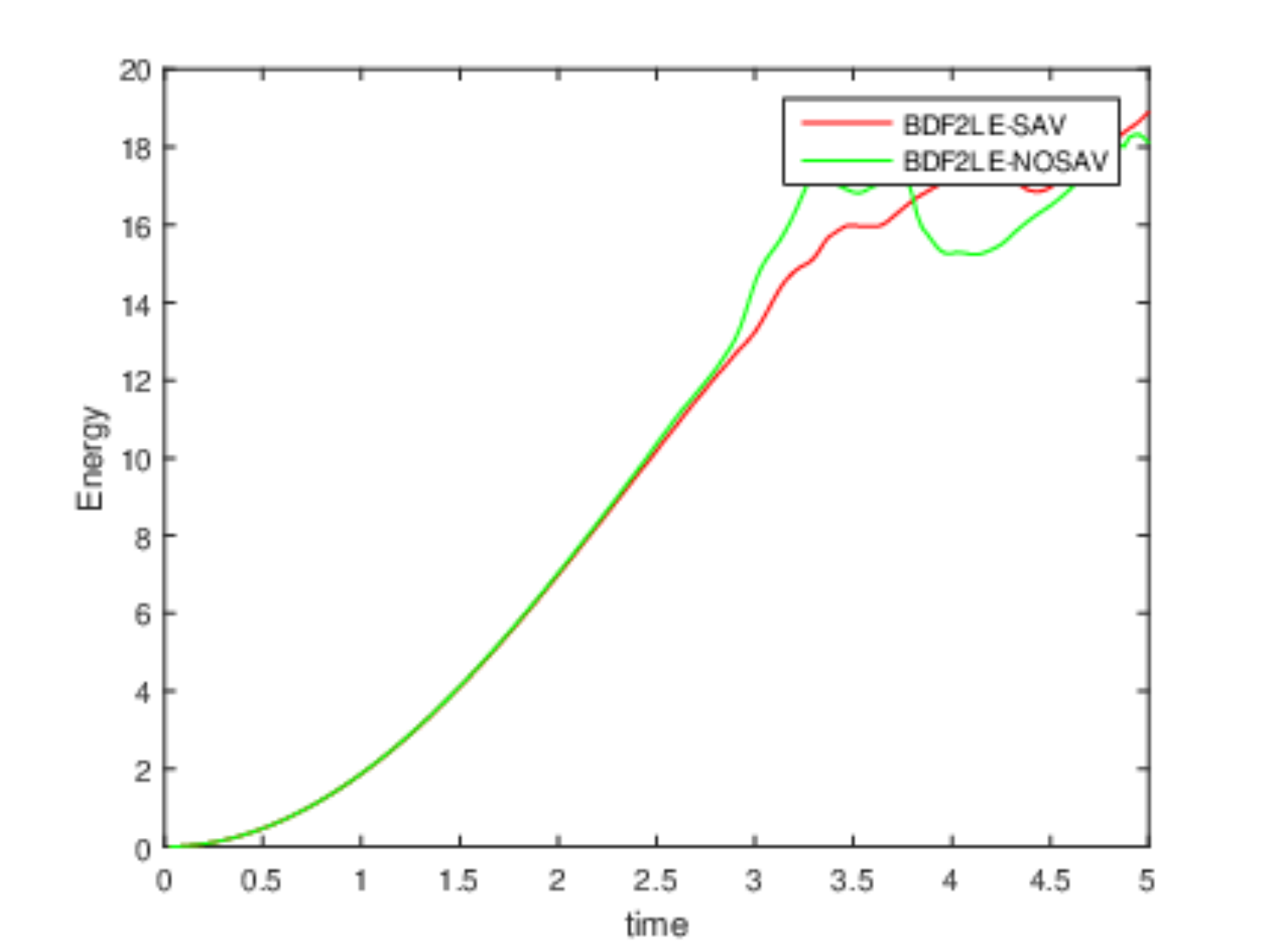}
\end{subfigure}
\caption{Time evolutions of energy for $Re=200,800,1200$ from left to right.}
	\label{fig:Energies}
\end{figure}
\begin{figure}[h!]
	\centering
	\begin{subfigure}[b]{0.3\linewidth}
		\includegraphics[width=55 mm]{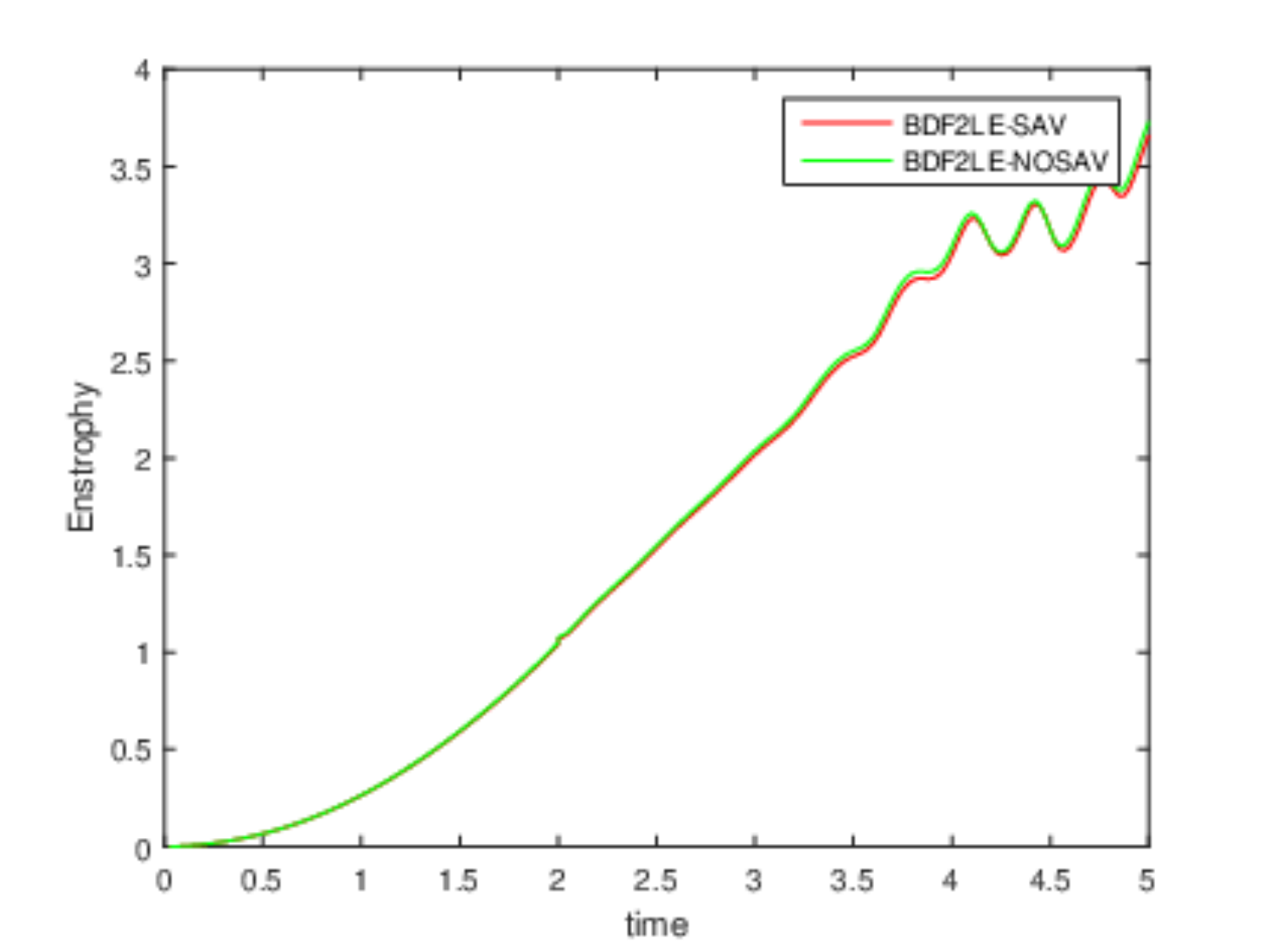}
	\end{subfigure}
	\begin{subfigure}[b]{0.3\linewidth}
		\includegraphics[width=55 mm]{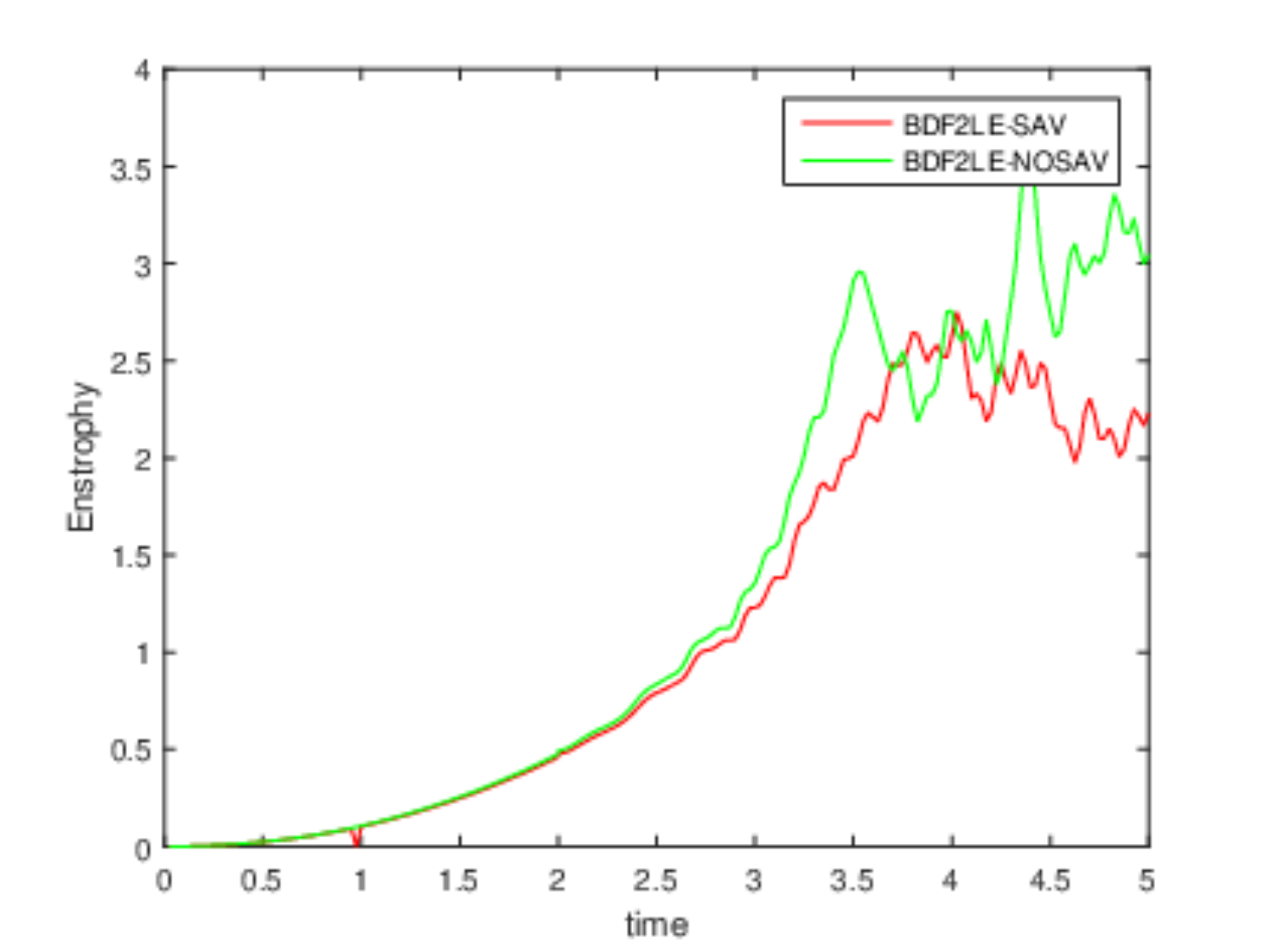}
	\end{subfigure}
	\begin{subfigure}[b]{0.3\linewidth}
		\includegraphics[width=55 mm]{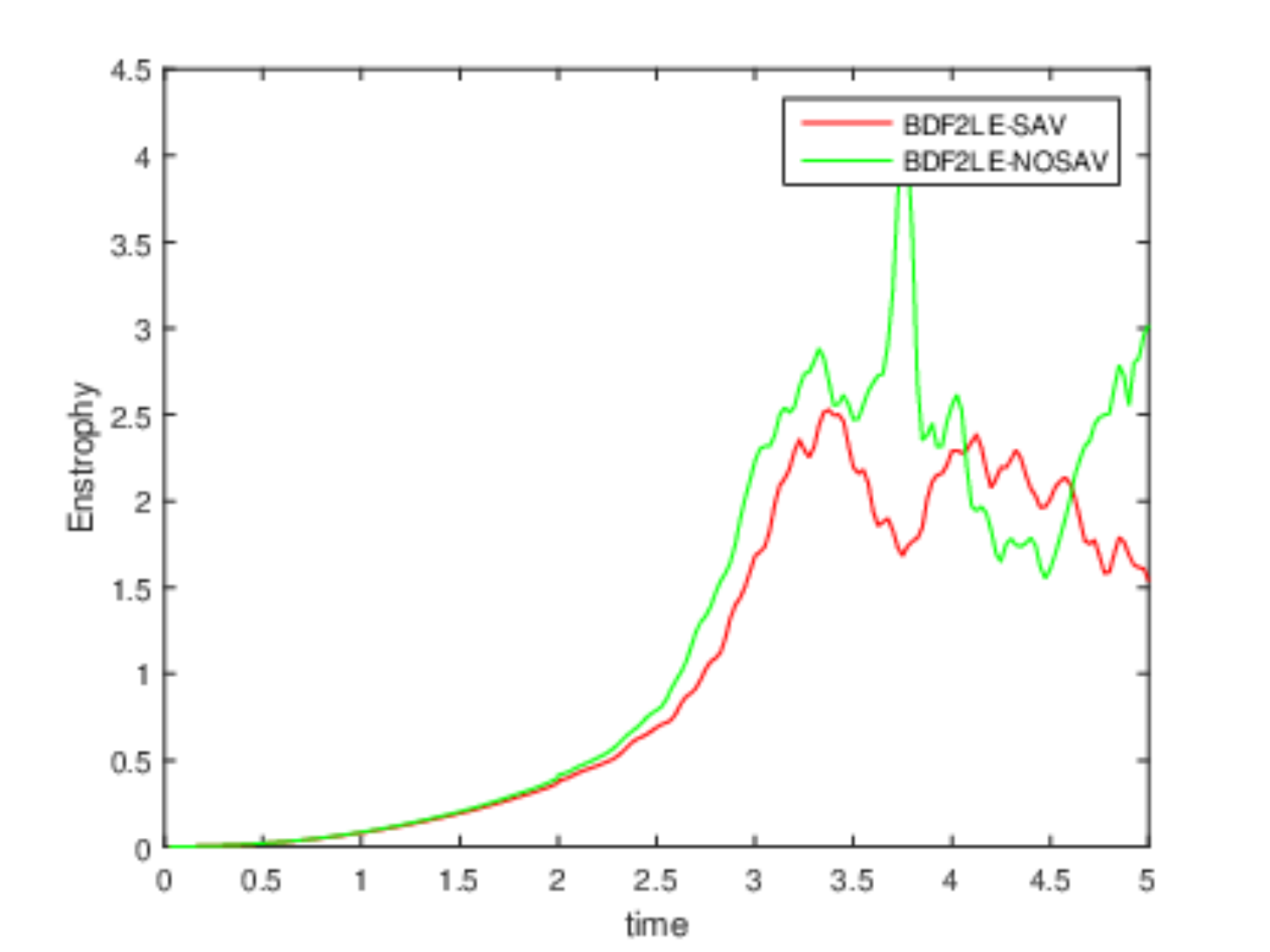}
	\end{subfigure}
	\caption{Time evolutions of enstrophy for $Re=200,800,1200$ from left to right. }
	\label{fig:Enstrophies}
\end{figure}
\section{Conclusion}
In this report, we introduced and analyzed SAV regularization method to find solutions to NSE with BDF2LE scheme. This stabilization is effective only for small scales in fluid flow. We have found that the solutions of the proposed algorithm preserve both energy and helicity identities. We have always obtained smooth and regular bounded solutions which don't require any condition on time step size. We also proved that the method is optimally convergent with suitable choices of the artificial viscosity and the grad-div stabilization parameter. Moreover, several numerical tests were performed verifying the theoretical findings and showed that the method provide better solutions over the unstabilized NSE. The incorporation of this type stabilization method into the framework for the design of adaptive algorithm is a subject of future research. In this way, the large scale space is determined posteriorly and it will be chosen differently on different mesh cells. 

\bibliographystyle{amsplain}

\bibliography{reference}
\begin{appendices}
\section{Error Analysis}

We now prove the statement of Theorem \ref{error}.
	\begin{proof}
		To obtain the error equation, denote $\bu(t^n)=\bu^n$. Then, the true solutions $(\bu_h^{n+1},p_h^{n+1},S_H^{n+1})$ at time level $t^{n+1}$ satisfy
		\begin{eqnarray}
		\bigg(\dfrac{3\bu^{n+1}-4\bu^n+\bu^{n-1}}{2\Delta t},\bv_h\bigg)+\nu\big(\nabla \bu^{n+1},\nabla \bv_h\big)
		+b\big(2\bu^{n}-\bu^{n-1},\bu^{n+1},\bv_h\big)\nonumber\\+\alpha_1\big(\nabla\times \bu^{n+1},\nabla\times \bv_h\big)-\alpha_1\big(\nabla\times \bu^{n+1},\nabla\times \bv_h\big)
		+\alpha_2\big(\nabla \cdot \bu^{n+1},\nabla\cdot \bv_h\big)\nonumber\\-\big(p^{n+1},\nabla\cdot\bv_h)
		=\big(f^{n+1},\bv_h\big)+Intp(\bu^{n+1};\bv_h)\label{er}
		\end{eqnarray}
		where the local truncation error is
		\begin{eqnarray}
		Intp(\bu^{n+1},\bv_h)=\bigg(\dfrac{3\bu^{n+1}-4\bu^n+\bu^{n-1}}{2\Delta t}-\bu_t^{n+1},\bv_h\bigg)+b\big(2\bu^{n}-\bu^{n-1},\bu^{n+1},\bv_h\big)-b\big(\bu^{n+1},\bu^{n+1},\bv_h\big).\nonumber
		\end{eqnarray}
		 Subtracting the equation  (\ref{vh1}) from (\ref{er}) yields
		\begin{eqnarray}
	\bigg(\dfrac{3\be^{n+1}-4\be^n+\be^{n-1}}{2\Delta t},\bv_h\bigg)+\nu(\nabla \be^{n+1},\nabla \bv_h)
			+b(2\bu^{n}-\bu^{n-1},\bu^{n+1},\bv_h)\nonumber\\
			-b(2\bu_h^{n}-\bu_h^{n-1},\bu_h^{n+1},\bv_h)
		+\alpha_1(\nabla\times \be^{n+1},\nabla \times \bv_h)	+\alpha_2(\nabla\cdot \be^{n+1},\nabla\cdot \bv_h)\nonumber\\=\alpha_1(\nabla\times \bu^{n+1},\nabla\times \bv_h)+(p^{n+1},\nabla\cdot \bv_h)\nonumber\\-\alpha_1(S_H^{n+1},\nabla \times \bv_h)+Intp(\bu^{n+1};\bv_h)\label{65}
		\end{eqnarray}
	Adding and subtracting terms for the  convective terms, by using the properties (\ref{d1}) and (\ref{d2}), one gets  
		\begin{eqnarray}
		&&b(2\bu^{n}-\bu^{n-1},\bu^{n+1},\bv_h)
			-b(2\bu_h^{n}-\bu_h^{n-1},\bu_h^{n+1},\bv_h)\nonumber\\
		&&=b(2\bu^{n}-\bu^{n-1},\bu^{n+1},\bv_h)
		-b(2\bu^{n}-\bu^{n-1},\bu_h^{n+1},\bv_h)\nonumber\\
		&&+b(2\bu^{n}-\bu^{n-1},\bu_h^{n+1},\bv_h)
		-b(2\bu_h^{n}-\bu_h^{n-1},\bu_h^{n+1},\bv_h)\nonumber\\
		&&=b(2\bu^{n}-\bu^{n-1},\be^{n+1},\bv_h)
		+b(2\be^{n}-\be^{n-1},\bu_h^{n+1},\bv_h).\label{ctrm}
		\end{eqnarray}
Decompose the velocity error in the usual way:
		\begin{eqnarray}
		\be^n=\big(\bu^n-I^h\bu^n\big)+\big(I^h\bu^n-\bu_h^n\big)= \bfeta^n+\bphi_h^n,\label{se}
		\end{eqnarray}
		where $I^h\bu^n$ is an interpolant of $\bu^n$ in $V_h$.\\
		Using error decomposition, $\bv_h=\phi_h^{n+1}$ in the (\ref{65}) and  (\ref{d1})-(\ref{d2}) yields
			\begin{eqnarray}
			\lefteqn{\dfrac{1}{4\Delta t}\bigg[\|\bphi_h^{n+1}\|^2+\|2\bphi_h^{n+1}-\bphi_h^{n}\|^2\bigg] -\dfrac{1}{4\Delta t}\bigg[\|\bphi_h^n\|^2+\|2\bphi_h^n-\bphi_h^{n-1}\|^2\bigg]+\dfrac{1}{4\Delta t}\|\bphi_h^{n+1}-2\bphi_h^n+\bphi_h^{n-1}\|^2}\nonumber\\
			&&+\nu\| \nabla \bphi_h^{n+1}\|^2
			+\alpha_1\|\nabla \times \bphi_h^{n+1}\|^2+\alpha_2\|\nabla \cdot \bphi_h^{n+1}\|^2\nonumber\\
			&=&-\bigg(\dfrac{3\bfeta^{n+1}-4\bfeta^n+\bfeta^{n-1}}{2\Delta t},\bphi_h^{n+1}\bigg)
			-\nu(\nabla \bfeta^{n+1},\nabla\bphi_h^{n+1})-\alpha_1(\nabla\times \bfeta^{n+1},\nabla\times \bphi_h^{n+1})\nonumber\\
			&&-\alpha_2(\nabla \cdot \bfeta^{n+1},\nabla\cdot \bphi_h^{n+1})-b(2\bu^{n}-\bu^{n-1},\bfeta^{n+1},\bphi^{n+1}_h)
			-b(2\bfeta^{n}-\bfeta^{n-1},\bu_h^{n+1},\bphi^{n+1}_h)\nonumber\\
			&&-b(2\bphi_h^{n}-\bphi_h^{n-1},\bu_h^{n+1},\bphi^{n+1}_h)+\alpha_1(\nabla\times \bu^{n+1}-S_H^{n+1},\nabla\times \bphi^{n+1}_h)\nonumber\\&&+(p^{n+1}-q_h,\nabla\cdot \bphi^{n+1}_h)+Intp(\bu^{n+1};\bphi^{n+1}_h)\label{d}
			\end{eqnarray}
The terms on the right hand side of (\ref{d}) have to be bounded. For the first term, applying Cauchy-Schwarz, Poincar\'e-Friedrichs inequalities, fundamental theorem of calculus and Young's inequality, one gets 
			\begin{eqnarray}
			\bigg|-\bigg(\dfrac{3\bfeta^{n+1}-4\bfeta^n+\bfeta^{n-1}}{2\Delta t},\bphi_h^{n+1}\bigg)\bigg|
			&\leq&\norm{\dfrac{3\bfeta^{n+1}-4\bfeta^n+\bfeta^{n-1}}{2\Delta t}}\norm{\nabla\bphi_h^{n+1}}\nonumber\\
			&\leq& \dfrac{C\nu^{-1}}{\Delta t}\int_{t^{n-1}}^{t^{n+1}}\|\bfeta_t\|^2dt+\dfrac{\nu}{16}\|\nabla \bphi_h^{n+1}\|^2 \label{t1}
			\end{eqnarray}
The next couple estimates will use Cauchy-Schwarz and Young's inequalities and they will also contribute to the error bound. One obtains in a straightforward way
			\begin{eqnarray}
			\big|-\nu(\nabla\bfeta^{n+1},\nabla\bphi_h^{n+1})\big|
			&\leq& C\nu\|\nabla\bfeta^{n+1}\|^2+\dfrac{\nu}{16}\|\nabla\bphi_h^{n+1}\|^2,\label{t5}
		\\
			\big|-\alpha_1(\nabla\times \eta^{n+1},\nabla\times \bphi_h^{n+1})\big|	&\leq& C\nu^{-1}\alpha_1^2\|\nabla \bfeta^{n+1}\|^2+\dfrac{\nu}{16}\|\nabla \bphi_h^{n+1}\|^2,\label{t6}
			\\
			\big|-\alpha_2(\nabla\cdot\bfeta^{n+1},\nabla\cdot \bphi_h^{n+1})\big|
			&\leq&C\alpha_2\|\nabla\bfeta^{n+1}\|^2+\dfrac{\alpha_2}{2}\|\nabla\cdot \bphi_h^{n+1}\|^2 .\label{t7}
			\end{eqnarray}		
We proceed to bound the convective terms using Cauchy-Schwarz, Poincar\'e-Friedrichs, Young's inequalities and the estimation (\ref{n3}):
			\begin{eqnarray}
			\big| -b(2\bu^{n}-\bu^{n-1},\bfeta^{n+1},\phi^{n+1}_h)\big|&\leq&C\|\nabla (2\bu^{n}-\bu^{n-1})\|\|\nabla \bfeta^{n+1}\|\|\nabla \bphi^{n+1}_h\|\nonumber\\
			&\leq& C \big(\|\nabla \bu^{n}\|+\|\nabla \bu^{n-1}\|\big)\|\nabla \bfeta^{n+1}\|\|\nabla \bphi^{n+1}_h\|\nonumber\\
			&\leq& C\nu^{-1}\|\nabla \bfeta^{n+1}\|^2\big(\|\nabla \bu^{n}\|^2+\|\nabla \bu^{n-1}\|^2\big)\nonumber\\
			&+&\dfrac{\nu}{16}\|\nabla \bphi^h_{n+1}\|^2, \label{t2}
			\end{eqnarray}
			and
			\begin{eqnarray}
			\big|-b(2\bfeta^{n}-\bfeta^{n-1},\bu_h^{n+1},\bphi^{n+1}_h))\big|
			&\leq& C\nu^{-1}\|\nabla \bu_h^{n+1}\|^2\big(\|\nabla \bfeta^{n}\|^2+\|\nabla  \eta^{n-1}\|^2\big)\nonumber\\
			&+&\dfrac{\nu}{16}\|\nabla \bphi^h_{n+1}\|^2.\label{t3}
			\end{eqnarray}
			Using (\ref{n2}), one gets for the last convective term
			\begin{eqnarray}
			\big| -b(2\bphi_h^{n}-\bphi_h^{n-1},\bu_h^{n+1},\bphi^{n+1}_h)\big|&\leq&\dfrac{1}{2}\bigg(\norm{2\bphi_h^{n}-\bphi_h^{n-1}}\|\nabla\bu_h^{n+1}\|_{\infty}\|\nabla \bphi^{n+1}_h\|\nonumber\\&&+\norm{2\bphi_h^{n}-\bphi_h^{n-1}}\norm{\bu_h^{n+1}}_{\infty}\|\|\nabla \phi^{n+1}_h\|\bigg)\nonumber\\
			&\leq& C\big(\| \bphi_h^{n}\|+\|  \bphi_h^{n-1}\|\big)\|\nabla \bphi_h^{n+1}\|\big(\|\nabla \bu^{n+1}_h\|_{\infty}+\norm{\bu_h^{n+1}}_{\infty}\big)\nonumber\\
			&\leq & C\nu^{-1}\big(\| \bphi_h^{n}\|^2+\|  \bphi_h^{n-1}\|^2\big)\big(\|\nabla \bu^{n+1}_h\|^2_{\infty}+\norm{\bu_h^{n+1}}_{\infty}^2\big)\nonumber\\&&+\dfrac{\nu}{16}\norm{\nabla\bphi_h^{n+1}}^2. \label{t4}
			\end{eqnarray}
		To bound the pressure term, we use the fact that $(\nabla\cdot\phi_h,q_h)=0$, $\forall q_h \in \bfV_h$ together with Cauchy-Schwarz and Young's inequalities:  
			\begin{eqnarray}
			\big|p(t^{n+1}),\nabla\cdot \bphi_h^{n+1})\big|
			&\leq& C\alpha_2^{-1}\big\|\inf_{q_h\in Q_h}\big\|p(t^{n+1})-q_h\big\|^2+\dfrac{\alpha_2}{4}\|\nabla\cdot \bphi_h^{n+1}\|^2.\label{t8}
			\end{eqnarray}
			Next, we bound the coarse mesh projection term. Using the definition of the $L^2$-projection operator $P_{L_H}$ $(\ref{cm})$ and from (\ref{vh3}), one can write $S_H^{n+1}=P_{L_H}(\nabla\times\bu^n_h)$. Then, we add and subtract  $P_{L_H}(\nabla\times\bu^n)$ and $\nabla\times\bu^n$ to the coarse mesh projection term and error definition, one gets
			\begin{eqnarray}
			\lefteqn{\big|\alpha_1\Big(\nabla \times \bu^{n+1}-S_H^{n+1},\nabla\times \bphi_h^{n+1}\Big)\big|}\nonumber\\
			&&= \big|\alpha_1 \big (P_{L_H}(\nabla\times \be^{n})+(I-P_{L_H})(\nabla\times \bu^{n})+(\nabla\times(\bu^{n+1}-\bu^n)),\nabla\times\bphi_h^{n+1}\big)\big|.\nonumber
		\end{eqnarray}
	Using error decomposition (\ref{se}), Cauchy-Schwarz and Young's inequalities, inverse estimation and the bound (\ref{5}) yields
		\begin{eqnarray}
		 \lefteqn{\big|\alpha_1 \big (P_{L_H}(\nabla\times \be^{n})+(I-P_{L_H})(\nabla\times \bu^{n})+(\nabla\times(\bu^{n+1}-\bu^n))\big),\nabla\times\bphi_h^{n+1}\big|}\nonumber\\
			&&\leq C\nu^{-1}\alpha_1^2\big\|P_{L_H}(\nabla\times\bfeta^n)\|^2 +C\nu^{-1}\alpha_1^2\|P_{L_H}(\nabla\times\bphi_h^n)\|^2\nonumber\\
			&&+C\nu^{-1}\alpha_1^2\|(I-P_{L_H})(\nabla\times \bu^{n})\|^2+C\nu^{-1}\alpha_1^2\|\nabla\times(\bu^{n+1}-\bu^n)\|^2+\dfrac{\nu}{16}\norm{\nabla \times \bphi_h^{n+1}}^2 \nonumber\\	
					&&\leq C\nu^{-1}\alpha_1^2\|\nabla\bfeta^n\|^2+C\nu^{-1}\alpha_1^2h^{-2}\|\bphi_h^n\|^2 +C\nu^{-1}\alpha_1^2\|(I-P_{L_H})(\nabla\times \bu^{n})\|^2\nonumber\\
			&&+C\nu^{-1}\alpha_1^2\|\nabla\times(\bu^{n+1}-\bu^n)\|^2+\dfrac{\nu}{16}\|\nabla\bphi_h^{n+1}\|^2.\label{t9}
			\end{eqnarray}
			Finally, the local truncation error $Intp(\bu^{n+1};\bphi_h^{n+1})$ can be bounded as follows. The first term of $Intp(\bu^{n+1};\bphi_h^{n+1})$ is estimated by using Cauchy-Schwarz, Poincar\'e  Friedrichs, Young's and H\"{o}lder's inequalities together with the integral remainder form of Taylor's theorem 
			\begin{eqnarray}
		\lefteqn{\bigg|-\bigg(\bu_t(t^{n+1})-\dfrac{3\bu(t^{n+1})-4\bu(t^n)+\bu(t^{n-1})}{2\Delta t},\bphi_h^{n+1}\bigg)\bigg|}\nonumber\\
			&\leq &\norm{\dfrac{3\bu(t^{n+1})-4\bu(t^n)+\bu(t^{n-1})}{2\Delta t}-\bu_t(t^{n+1})}\norm{\bphi_h^{n+1}}\nonumber\\
			&\leq& C\Delta t^{{3}}\nu^{-1}\int_{t_{n-1}}^{t_{n+1}}\|\bu_{ttt}\|^2dt+\dfrac{\nu}{16}\|\nabla \bphi_h^{n+1}\|^2. \label{t10}
			\end{eqnarray}
To bound the convective terms in $Intp(\bu^{n+1};\bphi_h^{n+1})$, we first rearrange the terms. Using the bound (\ref{n3}) and applying  Cauchy-Schwarz, Young's and H\"{o}lder's inequalities together with the Taylor's theorem with integral remainder, we get
			\begin{eqnarray}
			\lefteqn{b\big(2\bu^{n}-\bu^{n-1},\bu^{n+1},\bv_h\big)-b\big(\bu^{n+1},\bu^{n+1},\bv_h\big)}\nonumber\\
			&\leq&C\norm{\nabla(2\bu^{n}-\bu^{n-1}-\bu^{n+1})}\norm{\nabla\bu^{n+1}}\norm{\bphi_h^{n+1}}\nonumber\\
			&\leq& C\Delta t^{{3}}\nu^{-1}\norm{\nabla\bu^{n+1}}^2\int_{t_{n-1}}^{t_{n+1}}\|\nabla\bu_{tt}\|^2dt+\dfrac{\nu}{16}\|\nabla \bphi_h^{n+1}\|^2.\label{t11}
			\end{eqnarray} 
			Collecting all estimates (\ref{t1})-(\ref{t11}) and the equality (\ref{d}) simplifies to
			\begin{eqnarray}
	\lefteqn{	\dfrac{1}{4\Delta t}\bigg[\|\bphi_h^{n+1}\|^2-\|\bphi_h^n\|^2\bigg]+ \dfrac{1}{4\Delta t}\bigg[\|2\bphi_h^{n+1}-\bphi_h^{n}\|^2-\|2\bphi_h^n-\bphi_h^{n-1}\|^2\bigg]+\dfrac{1}{4\Delta t}\|\bphi_h^{n+1}-2\bphi_h^n+\bphi_h^{n-1}\|^2}\nonumber\\&&+\dfrac{7\nu}{16}\| \nabla \bphi_h^{n+1}\|^2 +\alpha_1\norm{\nabla\times\bphi_h^{n+1}}^2+\dfrac{\alpha_2}{4}\|\nabla \cdot \bphi_h^{n+1}\|^2\nonumber\\
			&\leq&\dfrac{C\nu^{-1}}{\Delta t}\int_{t^{n-1}}^{t^{n+1}}\|\bfeta_t\|^2dt+C\nu^{-1}\|\nabla \eta^{n+1}\|^2\big(\|\nabla \bu^{n}\|^2+\|\nabla \bu^{n-1}\|^2\big)\nonumber\\
			&&+ C\nu^{-1}\|\nabla \bu_h^{n+1}\|^2\big(\|\nabla \bfeta^{n}\|^2+\|\nabla  \bfeta^{n-1}\|^2\big)+ C\nu^{-1}\big(\| \bphi_h^{n}\|^2+\| \bphi_h^{n-1}\|^2\big)\big(\|\nabla \bu^{n+1}_h\|^2_{\infty}+\norm{\bu_h^{n+1}}_{\infty}^2\big)\nonumber\\
			&&+C\nu\|\nabla\bfeta^{n+1}\|^2+C\alpha_2^{-1}\big\|\inf_{q_h\in Q_h}\big\|p(t^{n+1})-q_h\big\|^2+C\nu^{-1}\alpha_1^2\|\big(\nabla\bfeta^n)\|^2\nonumber\\
			&&+C\nu^{-1}\alpha_1^2h^{-2}\|\bphi_h^n\|^2 +C\nu^{-1}\alpha_1^2\|(I-P_{L_H})(\nabla\times \bu^{n})\|^2+C\nu^{-1}\alpha_1^2\|\nabla\times(\bu^{n+1}-\bu^n)\|^2\nonumber\\
			&&+C\Delta t^{{3}}\nu^{-1}\int_{t_{n-1}}^{t_{n+1}}\|\bu_{ttt}\|^2dt+ C\Delta t^{{3}}\nu^{-1}\norm{\nabla\bu^{n+1}}^2\int_{t_{n-1}}^{t_{n+1}}\|\nabla\bu_{tt}\|^2dt .\label{A}
			\end{eqnarray}
	 Multiplication of each term by $4\Delta t$ and summation from $n=1$ to $n=N-1$ and using approximation properties (\ref{ap1})-(\ref{ap10}) and (\ref{pa}) to obtain
			\begin{eqnarray}
			\lefteqn{\|\bphi_h^{N}\|^2+ \|2\bphi_h^{N}-\bphi_h^{N-1}\|^2+\sum_{n=1}^{N-1}\|\bphi_h^{n+1}-2\bphi_h^n+\bphi_h^{n-1}\|^2}\nonumber\\
			&&+\Delta t\sum_{n=1}^{N-1}(\nu\norm{\nabla\bphi_h^{n+1}}^2+\alpha_1\norm{\nabla\times\bphi_h^{n+1}}^2+\alpha_2\norm{\nabla\cdot\bphi_h^{n+1}}^2)\nonumber\\
			&\leq&\|\bphi_h^1\|^2+\|2\bphi_h^1-\bphi_h^{0}\|^2+C\bigg(\nu^{-1}h^{2k+2}\norm{|\bu_t|}^2_{2,k+1}+\nu^{-1}h^{2k}\norm{|\bu|}^2_{2,k+1}\norm{|\nabla \bu|}^2_{\infty,0}+\nu h^{2k}\norm{|\bu|}^2_{2,k+1}\nonumber\\
			&&+\alpha_2^{-1}h^{2k+2}\norm{|p|}^2_{2,k+1}+\nu^{-1}\alpha_1^2h^{2k}\norm{|\bu|}^2_{2,k+1}+\nu^{-1}\alpha_1^2H^{2k}\norm{|\bu|}^2_{2,k+1}+\nu^{-1}\alpha_1^2\Delta t^2\norm{|\bu_t|}^2_{\infty,0}\nonumber\\
			&&+\nu^{-1}\Delta t^4\norm{|\bu_{ttt}|}^2_{2,0}+\nu^{-1}\Delta t^{{4}}\norm{|\nabla\bu|}^2_{\infty,0}\norm{|\nabla\bu_{tt}|}^2\bigg) +C\Delta t\nu^{-1}(1+\alpha_1^2h^{-2})\sum_{n=1}^{N-1}\norm{\bphi_h^n}^2.\nonumber
			\end{eqnarray}
 Lastly, the required result is proved by applying Lemma \ref{gl} and the triangle inequality to the splitting of the errors.
		\end{proof}
\end{appendices}

\end{document}